\documentclass[sn-mathphys-num]{sn-jnl}

\usepackage{graphicx}%
\usepackage{multirow}%
\usepackage{amsmath,amssymb,amsfonts}%
\usepackage{amsthm}%
\usepackage{mathrsfs}%
\usepackage[title]{appendix}%
\usepackage{xcolor}%
\usepackage{textcomp}%
\usepackage{manyfoot}%
\usepackage{booktabs}%
\usepackage{algorithm}%
\usepackage{algorithmicx}%
\usepackage{algpseudocode}%
\usepackage{listings}%
\usepackage[utf8]{inputenc}
\usepackage{mathtools, enumitem}

\newcommand{\pP}{\mathbb{P}}
\newcommand{\R}{\mathbb{R}}

\newcommand{\D}{\text{ d}}

\theoremstyle{thmstyleone}%
\newtheorem{theorem}{Theorem}[section]
\newtheorem{lemma}[theorem]{Lemma}

\theoremstyle{thmstyletwo}%
\newtheorem{remark}[theorem]{Remark}%

\theoremstyle{thmstylethree}%
\newcounter{assumption}
\newtheorem{assumption}[theorem]{Assumption}

\newcounter{subassumption}[assumption]
\renewcommand{\thesubassumption}{(\textit{\roman{subassumption}})}
\makeatletter
\renewcommand{\p@subassumption}{\theassumption}
\makeatother
\newcommand{\subasu}{
  \refstepcounter{subassumption}%
  \thesubassumption~\ignorespaces}

\usepackage[colorinlistoftodos,textsize=scriptsize]{todonotes}

\begin{document}

\title[A Cournot-Nash Model for a Coupled Hydrogen and Electricity Market]{A Cournot-Nash Model for a Coupled Hydrogen and Electricity Market}

\author[1]{\fnm{Pavel} \sur{Dvurechensky}}\email{pavel.dvurechensky@wias-berlin.de}

\author*[2]{\fnm{Caroline} \sur{Geiersbach}}\email{caroline.geiersbach@uni-hamburg.de}

\author[1,3]{\fnm{Michael} \sur{Hinterm\"uller}}\email{michael.hintermueller@wias-berlin.de}\email{hint@math.hu-berlin.de}

\author[3]{\fnm{Aswin} \sur{Kannan}}\email{aswin.kannan@hu-berlin.de}

\author[4]{\fnm{Stefan} \sur{Kater}}\email{stefan.kater@mathematik.uni-freiburg.de}

\author[5]{\fnm{Gregor} \sur{Z\"ottl}}\email{gregor.zoettl@fau.de}

\affil*[1]{\orgname{Weierstrass Institute}, \orgaddress{\street{Mohrenstrasse 39}, \city{Berlin}, \postcode{10117}, \country{Germany}}}

\affil[3]{\orgname{Universität  Hamburg}, \orgaddress{\street{Bundesstraße 55}, \city{Hamburg}, \postcode{20146}, \country{Germany}}}

\affil[3]{\orgname{Humboldt-Universit\"at zu Berlin}, \orgaddress{\street{Unter den Linden 6}, \city{Berlin}, \postcode{10099},  \country{Germany}}}

\affil[4]{\orgname{Albert-Ludwigs-Universität Freiburg i. Br.}, \orgaddress{\street{ 	Hermann-Herder-Str. 10}, \city{Freiburg}, \postcode{79104}, \country{Germany}}}

\affil[5]{\orgname{Friedrich-Alexander-Universität Erlangen-Nürnberg}, \orgaddress{\street{Schloßplatz 4}, \city{Nürnberg}, \postcode{90020}, \country{Germany}}}


\abstract{We present a novel model of a coupled hydrogen and electricity market on the intraday time scale, where hydrogen gas is used as a storage device for the electric grid. Electricity is produced by renewable energy sources or by extracting hydrogen from a pipeline that is shared by non-cooperative agents. The resulting model is a generalized Nash equilibrium problem. Under certain mild assumptions, we prove that an equilibrium exists. Perspectives for future work are presented.}




\maketitle

\section{Introduction}

One of the central challenges in variable renewable energy (VRE) sources like solar and wind energy on a large scale is in its variability due to fluctuations in weather and seasons. To balance energy supply and demand in real time, a storage mechanism is needed. One strategy in transitioning from fossil fuels is power-to-X, where power from renewable energy sources is transformed into other fuels that can be later used to produce electricity. Hydrogen (H$_2$) gas is one such fuel that is being considered as part of the energy transition. H$_2$ can be produced by electrolysis using VRE and can be stored in pipelines and caverns much as natural gas is stored today. During periods of low VRE production, this gas can be used to produce electricity. Hydrogen gas produced by VRE is considered ``green'' since hydrogen can be produced without releasing CO$_2$.

In this contribution, we present a novel mathematical model for a multi-modal energy system whereby a hydrogen network is coupled with an electricity grid. Non-cooperative agents can produce or sell different energy sources in a strategic way in order to maximize profit. We account for a shorter time horizon and uncertainty coming from fluctuations in demand. Our model aims to fill a gap in the game-theoretic literature, since so far, research has investigated models for electricity and gas markets separately but not yet as a coupled system. 
Electricity market models are based on the instantaneous transport of electricity; conventionally, they are oriented towards short-term planning over discrete time intervals. Gas networks, on the other hand, operate much more slowly and the physics of gas transport is governed by partial differential equations (PDEs) defined over continuous time. When the markets are coupled, it is natural to use a continuous, multiscale time model. Some seminal works have considered continuous time electricity market models and differential variational inequalities~\cite{terry16dvi,reeto08}. In the interest of coupling and physical practicalities, we adhere to this framework. We mention that~\cite{fokken21} is one recent work that studies a setting with both electrical and gas networks. However, the focus of this work is on solving a supply-demand problem, which includes constraints from the electrical network, constraints from the gas network in the form of partial differential equations, and coupling constraints from conversion. The setting does not focus on multiple agents or a game-theoretic problem. 
To our knowledge, there are no other works that focus on such joint markets in a game-theoretic setting.  


Our work contributes to several strands of the literature. In the context of energy  markets, we find several contributions that consider strategic interactions of firms in the
presence of shared and static, i.e., stationary, network constraints; see, e.g.,
\cite{SchiroPang:2013, HolmnbergandPhilpott:2018}.
Other contributions consider strategic market interaction in a dynamic context---but in the absence of
shared network constraints. The paper \cite{JunandVives:2004} considers adjustment costs when changing their output. Another work \cite{LedvinaandSircar:2011} studies strategic firms exploiting a private resource stock, and in \cite{LAMBERTINIPalestrini:2014}, the  authors consider the accumulation of private productive capacities over time.

Several works on dynamic oligopoly models do consider
shared constraints. In one strand of literature, the shared constraint can be
represented by a single state variable governed by an ordinary
differential equation. In \cite{XinandSun:2018,COLOMBOLabrecciosa:2019}, 
 differential games where firms jointly exploit a common renewable
production asset are studied. In contrast to those contributions, for a proper modeling of the physics underlying 
gas transport in our setup, we have to consider a spatially distributed state variable governed by a system of PDEs. Generalized Nash equilibrium problems (GNEPs) involving shared constraints in the form of a PDE include \cite{hintermueller2013pde,
Hintermueller2015,
DrevesGwinner:2016,
Gugat2018string,
kanzow2019multiplier}. None of these recent contributions, however, involve a model with a coupled electricity/gas market. A central difficulty, from a modeling and theoretical perspective, is the conversion between the two commodities and their different time scales. To make the two systems compatible, we adopt a continuous-time model for the electricity market. Moreover, the analysis of the equilibrium problem is delicate, requiring an understanding of the underlying PDE describing gas transport.


The proposed model is detailed in Section~\ref{sec:model}. 
In Section~\ref{sec:existence}, we prove the main theoretical result, namely, that the model is well-posed in the sense that an equilibrium exists (under certain mild assumptions). This work will be used as the foundation of future studies, which will be further described in Section~\ref{sec:conclusion}. 

\section{Equilibrium model}
\label{sec:model}
      

    

In this section, we describe our mathematical model of an energy market connected to physics on a coupled hydrogen and electricity network. Throughout, the topology of the latter is considered fixed over a given time horizon $T>0$. More specifically, our model comprises $N$ non-cooperative agents (firms) that seek to maximize their respective profit by making strategic decisions at the nodes of the energy network in the time period $[0,T]$. The general situation we want to analyze is as follows. Firms serve customers  who want to consume both electricity and hydrogen (both given in the units of megawatt hours, or MWh) at different locations. Firms own renewable electricity-generation facilities at different locations. Electricity can be either sold to satisfy demand of customers or it can be transformed into hydrogen. Hydrogen can be sold to customers or it can be transformed back into electricity. 

Our mathematical description of the model will proceed as follows: first, we define the network topology in Section~\ref{sec:network}. We then detail the possible decisions available to an agent and the agent's objective function in Section~\ref{sec:agents objective}. Operational constraints on the electricity and gas network are presented in Section~\ref{sec:cons-elec}. In Section~\ref{sec:summary-model}, we summarize the full model.

\subsection{Network topology}
\label{sec:network}

The electricity grid and the gas network are given by the finite, directed, connected, and acyclic graphs $\mathcal{G}^E = (\mathcal{V}^E, \mathcal{E}^E)$ and $\mathcal{G}^H = (\mathcal{V}^H, \mathcal{E}^H)$, respectively. These are subgraphs of the directed graph $\mathcal{G} = (\mathcal{V}, \mathcal{E}) = (\mathcal{V}^E \cup \mathcal{V}^H, \mathcal{E}^E \cup \mathcal{E}^H)$ representing the entire physical infrastructure, where $\mathcal{V}$ is the set of nodes/vertices and $\mathcal{E}$ denotes the set of edges. The nodes in $\mathcal{V}^E \cap \mathcal{V}^H$ are partitioned into the disjoint sets $\mathcal{V}^{\textup{GtP}}$ and $\mathcal{V}^{\textup{PtG}}$ where either electricity is converted to gas (PtG) or gas is used to generate electricity (GtP)\footnote{Reversible PtG stations, where systems also allow for GtP as discussed in \cite{Glenk2022}, can be accommodated by splitting the node at one location into two nodes and adding edges.}.
We also distinguish dedicated sale nodes $\mathcal{V}^s$, where one of the two commodities, electricity or hydrogen, is sold. GtP nodes, 
together with nodes where renewable energy is used to generate electricity, are combined in $\mathcal{V}^g$. By convention,  
only one plant/station is located at each node. 
%
The above network components are summarized in Table~\ref{table:notation_network}. A graphic depicting the coupled network is shown in Figure~\ref{fig:networks} with the corresponding graph shown in Figure~\ref{fig:networks-abstract}.

\begin{table}[]
    \centering
    \begin{tabular}{|l|l|}
        \hline
        Set & Description\\
        \hline
    $\mathcal{V}^E$  &  Electricity nodes\\
        $\mathcal{V}^H$    &  Hydrogen nodes\\
    $\mathcal{V}^{\textup{PtG}}$& Power-to-gas stations\\
        $\mathcal{V}^{\textup{GtP}}$& Gas-to-power stations\\
    $\mathcal{V}^s$ & Sale nodes\\
    $\mathcal{V}^g$ & Generation nodes\\
    $\mathcal{V}^H_{\partial}$ & Boundary hydrogen nodes $((\mathcal{V}^s\cap \mathcal{V}^H)\cup \mathcal{V}^{\textup{GtP}} \cup \mathcal{V}^{\textup{PtG}}) $  \\
   $\mathcal{V}^H_0$ & Interior hydrogen nodes $(\mathcal{V}^H \setminus \mathcal{V}^H_{\partial})$\\
        \hline
    \end{tabular}
    \caption{Description of nodal sets used in the network graph.}
    \label{table:notation_network}
\end{table}
\begin{figure}
\centering\includegraphics[height=8cm]{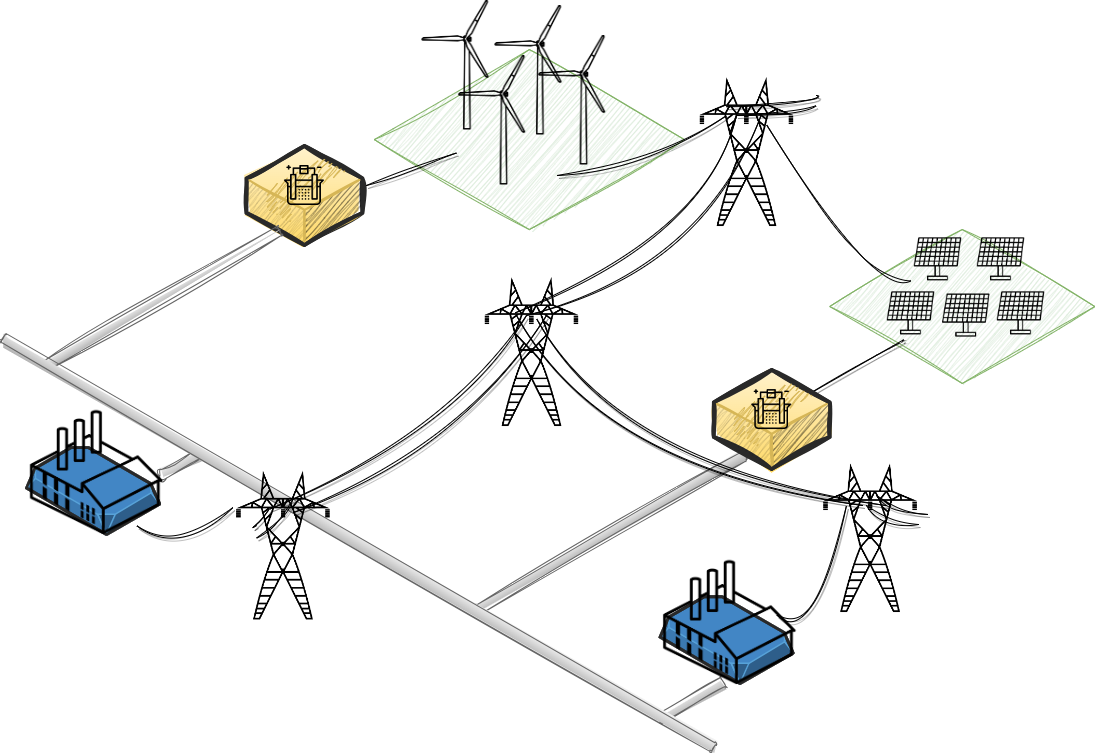}
    \caption{A coupled hydrogen and electricity network, where renewable energy sources are connected to the power grid for immediate use and some power can be used to power PtG stations, which are connected to the gas pipeline. GtP stations are depicted in the lower left and right, which uses hydrogen to produce power. In this system, hydrogen can be stored in the pipeline for later use. PtG and GtP stations are located at conversion nodes. The power lines are mathematically represented by edges $\mathcal{E}^E$ of the graph $\mathcal{G}^E$. The pipes in the gas pipeline are represented by edges $\mathcal{E}^H$ in the graph $\mathcal{G}^H.$}
    \label{fig:networks}
\end{figure}

\begin{figure}
    \centering
    \includegraphics[height=8cm]{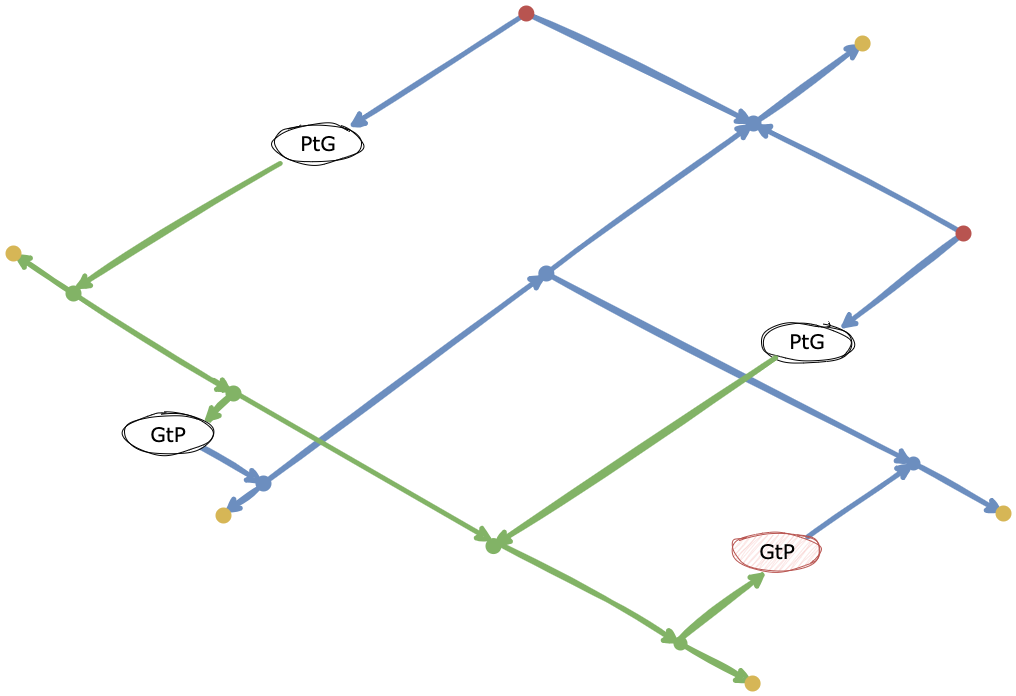}
    \caption{Representation of Figure~\ref{fig:networks} into abstract components. PtG/GtP nodes connect the hydrogen network (in green) and electricity network (in blue). Sale and generation nodes are highlighted in yellow and red, respectively. Arrows represent direction of flow.}
    \label{fig:networks-abstract}
\end{figure}


\subsection{Agent's decisions and objective}
\label{sec:agents objective}

For our high-fidelity model, we confine the time horizon $T$ and consider a short-term situation, which focuses on the operation of already existing  production 
and conversion facilities owned by the agents. At the time of making their decisions,  all cost aspects related to long run investment are already sunk and only short-term variable cost of operation are relevant for firms' decisions.  
\begin{table}[]
    \centering
    \begin{tabular}{|c|c|c|c|}
    \hline
    Variable & Type & Location & Dimension \\
    \hline
        $g_i$ & generation & $\mathcal{V}^g$ & $n_g$\\
        $s_i$ & sale & $\mathcal{V}^s$ &$n_s$\\
        $c_i$ & conversion & $\mathcal{V}^{\textup{PtG}}$ & $n_c$ \\
        $\hat{p}_i$ &pressure & $\mathcal{V}_\partial^{H}$ &$n_h$\\
        \hline
    \end{tabular}
    \caption{Decision variables for agent $i$ with their location. The dimension corresponds to the total number of nodes in the respective set.}
    \label{table:decision_variables}
\end{table}
\paragraph{Decisions.}
At any point $t\in [0,T]$, firms can make different strategic decisions to maximize their respective revenue. In our setting, the decision variables are ``generation,'' ``sale,''
``conversion,'' and ``pressure,'' respectively; see Table~\ref{table:decision_variables} for a list of these variables along with their ``locations,'' i.e., the associated set of nodes. 
\begin{itemize}
    \item \textbf{Generation}: $g_{i\nu}(t)$ is the rate at which a quantity (in MWh) of electricity generated by agent $i$ at node $\nu \in \mathcal{V}^g$ at time $t$ and $g_i(t) = (g_{i\nu}(t))_{\nu \in \mathcal{V}^g}$ is the vector of the power generated at all generation nodes. All quantities considered as a vectorial function over the time interval satisfy  $g_i \in H^1(0,T)^{n_g}$~\footnote{Here, the Sobolev space $H^1(0,T)$ denotes the set of functions in the Lebesgue space $L^2(0,T)$ with weak derivative in $L^2(0,T)$; see \cite{Adams2003} for a general definition and further properties of these spaces. Its $k$th product is denoted by $H^1(0,T)^{k} = \prod_{j=1}^{k} H^1(0,T)$. The function space regularity is dictated by the regularity needed for flow at the boundary; see also Section~\ref{sec:existence}. }.  A generation node in $\mathcal{V}^g$ can be either a renewable solar/wind plant or a GtP station.
    \item \textbf{Sale}: $s_{i\nu}(t)$ is the rate at which a quantity (in MWh) of electricity or hydrogen is sold by agent $i$ at $\nu \in \mathcal{V}^s$ at time $t$ and $s_i(t) = (s_{i\nu}(t))_{\nu \in \mathcal{V}^s}$ is the vector of all its sales. As before, we have $s_i \in H^1(0,T)^{n_s}$. Sales of a given commodity are made on dedicated sale nodes $\mathcal{V}^s$.
    \item \textbf{Conversion}: $c_{i\nu}(t)$ is the rate at which a quantity (in MWh) of electricity is used by agent $i$ at a PtG node $\nu \in \mathcal{V}^{\textup{PtG}}$ at time $t$. These quantities are collected in the vector $c_{i}(t) = (c_{i\nu}(t))_{\nu \in \mathcal{V}^{\textup{PtG}}}$. Again, we have $c_i \in H^1(0,T)^{n_c}$. 
    \item \textbf{Pressure}: $\hat{p}_{i\nu}(t)$ is the pressure (in Pa=$N/m^2$) dictated by agent $i$ at boundary nodes $\mathcal{V}_\partial^H$ at time $t$. The pressures are collected in the vector $\hat{p}_{i}(t)=(\hat{p}_{i\nu}(t))_{\nu \in \mathcal{V}_\partial^{H}}$. The function $\hat{p}_i(\cdot)$ is assumed to be in $H^1(0,T)^{n_h}$, where $n_h$ is the total number of nodes in $\mathcal{V}_{\partial}^H$. 
  \end{itemize}

As we will show in Lemma~\ref{lemma:regularity-PDE}, the transport of gas in the network is entirely dictated by the pressures and mass flows given at the boundary nodes $\mathcal{V}^H_{\partial}$. Mass flow, however, can be eliminated as a decision variable. Let $\eta_\nu$ be a measure of efficiency at station $\nu$ in the conversion from MWh of electricity to the equivalent quantity of gas in kg. We impose the following conditions from conversion/generation/sales to mass flow:
\begin{equation}
\label{eq:conversion-conditions}
\hat{q}_{i\nu}(t) =   
\begin{cases}
\eta_{\nu} c_{i\nu}(t), \quad & \text{if } \nu \in \mathcal{V}^{\textup{PtG}},\\
\eta_{\nu} g_{i\nu}(t), & \text{if } \nu \in \mathcal{V}^{\textup{GtP}},\\
\eta_{\nu}  s_{i\nu}(t), & \text{if } \nu \in \mathcal{V}^s \cap \mathcal{V}^H.
    \end{cases}
\end{equation}
As with the pressure variable, we write $\hat{q}_{i}(t)=(\hat{q}_{i\nu}(t))_{\nu \in \mathcal{V}_\partial^{H}}$, and $\hat{q}_i(\cdot)$ is assumed to be a function in $H^1(0,T)^{n_h}$. 

The full set of decisions made by agent $i$ is
\begin{equation}
\label{eq:definition-decision-space}
u_i = (g_i,s_i,c_i,\hat{p}_i) \in U_i:= H^1(0,T)^{n_g} \times H^1(0,T)^{n_s} \times H^1(0,T)^{n_c} \times  H^1(0,T)^{n_h}. 
\end{equation}
We further collect the decisions of all $N$ agents in the vector 
\[u = (u_1, \dots, u_N) \in U:=\prod_{i=1}^N U_i,\]
and use the notation $u_{-i}=(u_1, \dots, u_{i-1}, u_{i+1}, \dots, u_N)$ to denote the vector of decisions formed by all agents except the $i$th one as well as $(u_i, u_{-i}) = (u_1, \dots, u_N)$ to emphasize agent $i$ in relation to the other agents. Naturally, each agent's decisions are bounded. To accommodate this, we impose box-type constraints on the associated decision variables to be satisfied pointwise in time.
Note, moreover, that our setup models short-term interaction of firms in a network environment for a comparatively short horizon of time given by $[0,T]$. Clearly, before and after the considered horizon, the networks are supposed to be regularly operated. Our model does not explicitly provide a prediction for what happens beyond the time horizon considered. To avoid unrealistic outcomes, such as a complete depletion of all hydrogen contained in the network at the end of the horizon considered, we make the following additional assumption, which requires that firms over the entire horizon cannot extract more from the network than they inject:
\begin{equation}
\label{eq:decision-variable-restriction}
 \int_0^T \Big( \sum_{\nu \in \mathcal{V}^{\textup{PtG}}}  q_{i\nu}(\tau) - \sum_{\nu' \in \mathcal{V}^{\textup{GtP}}} q_{i\nu'}(\tau) - \sum_{\nu'' \in \mathcal{V}^{s}\cap \mathcal{V}^H} q_{i\nu''}(\tau) \Big)\D \tau  \geq 0 \;\; \forall i \in \left\{1,\hdots,N \right\}.
\end{equation}

In sum, private restrictions on the agents' decisions are captured by the constraints
\begin{equation}
\label{eq:Uad}
\begin{aligned}
    u_i \in U_{\textup{ad}, i} \coloneqq \{(g_i,s_i,c_i,\hat{p}_i) \in U_i \mid \, &\eqref{eq:decision-variable-restriction}\text{ satisfied and } \\
    & 0 \leq g_{i\nu}(t)\leq g_\nu^{\max} \, \forall \nu \in \mathcal{V}^g,  \\
    &0 \leq s_{i\nu}(t) \leq s_\nu^{\max}\, \forall \nu \in \mathcal{V}^s, \\
    & 0 \leq c_{i\nu}(t) \leq c_\nu^{\max} \,\forall \nu \in \mathcal{V}^{\textup{PtG}} \\
    & 0 \leq \hat{p}_{i\nu}(t) \leq p_\nu^{\max} \,\forall \nu \in \mathcal{V}^H_\partial    \,\,  \text{ a.e. }t \in [0,T] \}.
\end{aligned}
\end{equation}
We also define $U_{\rm ad}:=\prod_{i=1}^N U_{\textup{ad}, i}$.

\paragraph{Objective.}
We consider here a market setup in which agents generate revenue according to a Cournot--Nash model. Firms are supposed to face a common market demand both for electricity and hydrogen. Demand is supposed to be uncertain over time, which we model with the help of the complete probability space $(\Omega, \mathcal{F}, \pP)$, where $\Omega$ is the sample space, $\mathcal{F} \subset 2^{\Omega}$ is the $\sigma$-algebra of possible events, and $\pP\colon \Omega \rightarrow [0,1]$ is a probability measure (which is known by all agents).
Given a time $t \in [0,T]$ and a scenario $\omega \in \Omega$ of possible demand, the price $\pi_\nu$ of the corresponding commodity at node $\nu \in \mathcal{V}^{s}$ is a stochastic process 
given by the parameterized (by $\omega$) inverse demand function 
\begin{equation}
\label{eq:inverse-demand-function}
    \pi_\nu(t,\omega) := a_\nu(t,\omega) - b_\nu(t,\omega) \sum_{i=1}^N s_{i\nu}(t).
\end{equation}

The intercept $a_\nu(t,\omega)\in [0,a^{\max}]$  and slope $b_\nu(t,\omega) \in [0,b^{\max}]$ are dependent on the node $\nu$, so that prices of the commodities (possibly depending on their strategic position in the network) can be modeled separately. 
The parameter $a_\nu(t,\omega)$ can be interpreted as a per-unit reserve price from which on customers start consuming at all. For larger quantities, customers are willing to pay a lower per-unit price. For large quantities, when reaching a satiation point, this willingness to pay becomes zero. This is governed by the parameter $b_\nu(t,\omega)$ in our setup. Both parameters are assumed to subject to uncertainty $\omega$. That is, for different realizations of $\omega$ different reserve prices and different satiation points might occur.

This setup for our inverse demand functions is based on quasi-linear preferences; consumption choices (and thus also inverse demand) are independent of the specific financial endowment of consumers. This is a very standard approach to model strategic interactions in markets (see \cite{whinston1995microeconomic}, also for some fundamental notions on inverse demand functions).
Moreover, our specification of market prices depending linearly on demand (inducing
quadratic consumer surplus in the objective functions) is very common in the literature, providing computational approaches for solving equilibrium problems in
energy markets; see, e.g., \cite{chyong2014strategic,egging2013benders,yang2016exploration} for some examples.
We denote the short-term variable (nonnegative) cost to agent $i$ associated with an extra unit of electricity produced or converted at node $\nu$ by $\kappa_{i\nu}$. Typically those short-term costs are small, so many studies (for more details see  e.g.~\cite{kostFraunhofer2024} and ~\cite{shahid17}) disregard them, but we include them to make our setup more flexible. We, however, exclude intertemporal dynamic aspects such as maintenance-cycles, startup cost, and ramping costs possibly associated with conversion between electricity and hydrogen, see e.g.~\cite{shahid17}.


The expected profit of agent $i$ over the entire time period $[0,T]$ given the sales ($s_i,s_{-i})$ is then 
\begin{equation*}
\label{eq:parametrized-objective-function}
   \int_\Omega  \int_0^T \left( \sum_{\nu \in \mathcal{V}^s}   s_{i\nu}(\tau) \pi_\nu(\tau,\omega) - \sum_{\nu' \in \mathcal{V}^g} g_{i\nu'}(\tau) \kappa_{i\nu'}     -\sum_{\nu'' \in \mathcal{V}^{\textup{PtG}}}  c_{i\nu''}(\tau)  \kappa_{i\nu''}   \right) \D \tau \D \pP(\omega).
\end{equation*}
By linearity and essential boundedness of the coefficients $a(t,\omega)$ and $b(t,\omega)$, the above objective simplifies since
\begin{align*}
&\int_\Omega \int_0^T \sum_{\nu \in \mathcal{V}^s}   s_{i\nu}(\tau) \pi_\nu(\tau,\omega) \D t \D \pP(\omega) = \int_0^T \sum_{\nu \in \mathcal{V}^s}   s_{i\nu}(\tau) \int_\Omega \pi_\nu(\tau,\omega) \D \pP(\omega) \D t\\
& \quad =  \int_0^T \sum_{\nu \in \mathcal{V}^s}   s_{i\nu}(\tau) \int_\Omega \left(a_\nu(\tau,\omega) - b_\nu(\tau,\omega) \sum_{i=1}^N s_{i\nu}(\tau)\right) \D \pP(\omega) \D t\\
& \quad = \int_0^T \sum_{\nu \in \mathcal{V}^s}   s_{i\nu}(\tau) \left(\bar{a}_\nu(\tau) - \bar{b}_\nu(\tau) \sum_{i=1}^N s_{i\nu}(\tau)\right)  \D t,
\end{align*}
where $\bar{a}_\nu(t) := \int_\Omega a_\nu(t,\omega) \D \pP(\omega)$ is the average intercept and $\bar{b}_\nu(t):= \int_\Omega b_\nu(t,\omega)\D \pP(\omega)$ is the average slope. In other words, the objective can be expressed as a deterministic one.
To summarize, the agent's objective function is defined by
\begin{align*}
 f_i(u_i,u_{-i}) :=  \int_0^T  & \left[\sum_{\nu \in \mathcal{V}^s}   s_{i\nu}(\tau) \left(\bar{a}_\nu(\tau) - \bar{b}_\nu(\tau) \sum_{i=1}^N s_{i\nu}(\tau)\right) \right.  \\
 &\left. \quad -\sum_{\nu' \in \mathcal{V}^g} g_{i\nu'}(\tau) \kappa_{i\nu'}     -\sum_{\nu'' \in \mathcal{V}^{\textup{PtG}}}  c_{i\nu''}(\tau)  \kappa_{i\nu''} \right]\D \tau.
\end{align*}
For simplicity, we write the objective as being defined on $U$.

As a remark, we assumed the affine linear structure \eqref{eq:inverse-demand-function} for the inverse demand function, but in principle, this could be generalized to nonlinear demand functions. For a result given later in Theorem~\ref{thm:NI-function-lsc}, it would suffice to consider concave demand functions.
We mention that there is extensive literature on standard one-shot Cournot setups in the absence of any network constraints, which establishes more general assumptions required on inverse demand functions to guarantee existence and uniqueness of the market equilibrium (see e.g. \cite{vives1999oligopoly}). This could provide guidance on how to further extend these results to our setup involving network constraints.

\subsection{Operational constraints}
\label{sec:cons-elec}

In this section, we detail constraints on the operations of the electricity network and gas network, respectively. In our framework, agents are assumed to directly respect all constraints resulting from network flows. As agents share the same electricity and gas network, the subsequent constraints are shared constraints of the resulting GNEP stated in Section~\ref{sec:summary-model}.

Note that in  liberalized energy markets, a transmission system operator (TSO) with its own rules and goals is typically responsible for the proper functioning of network transport.  The TSO enforces a so-called congestion management regime, which imposes financial incentives and technical limits on the behavior of market participants. Real congestion management regimes organized by a TSO do not ideally reflect the physical network constraints, but induce some kind of inefficiencies. In contrast, our setup refrains from explicitly modelling the TSO and its specific objectives. This corresponds to a situation in which an ideal congestion management regime  perfectly conveys transmission constraints to market participants.


\paragraph{Electricity network} 
On both networks, we assumed that the underlying graph is finite, directed, connected, and acyclic. Finite and connected networks are standard in electricity market literature~\cite{liu04ptdf,Wood1996PowerGO}. Some seminal works in electricity market literature~\cite{notes-anthony21,anthony17stochastic} also use directed graphs; bidirectional flows can, however, be incorporated into our model by adding edges.  Agents can generate, sell, or convert electricity. On the power network $\mathcal{G}^E$, we impose \textit{generation-conversion-sales balance} for each agent, i.e., for $i=1,\ldots, N$, we require
\[
  \sum_{\nu \in \mathcal{V}^E} (s_{i\nu}(t) + c_{i\nu}(t)-g_{i\nu}(t)) = 0 \quad  \text{a.e.~in } (0,T),
\]
meaning that the sum of all power generated by an agent must be immediately sold or converted to hydrogen. Additionally, we impose \textit{transmission limits}, i.e., for every power line $e\in\mathcal{E}^E$, we have
 \[  \sum_{\nu \in \mathcal{V}^E} \hat{A}_{e\nu} \sum_{i =1}^N (s_{i\nu}(t) + c_{i\nu}(t)-g_{i\nu}(t)) \leq \theta_e \quad  \text{a.e.~in } (0,T) \]
 for given capacity limits $\theta_e>0$. (The above constraint becomes double sided when the electrical network is represented by an undirected graph). Note that $\hat{A}_{e\nu}$ represents the linear load flow approximation or the distribution factor matrix~\cite{liu04ptdf} (say, injection shift factor). This refers to the amount of power flow induced in line $e$ of the network by unit injection/extraction of power at node $\nu$. The power distribution factor that we use in our model represents a DC-based approximation to load flow. This substitutes a combination of Kirchoff's energy and voltage laws~\cite{Wood1996PowerGO,liu04ptdf} under the assumption that the voltage phase angles play a minimal role. The distribution factor matrix takes the form
$$\hat{A}_{e\nu} = B M_a\bar{B}^{-1}, \hspace{2mm} \bar{B} = M_a^\top B M_a, \hspace{2mm} \textrm{where}$$
$$M_a = \begin{pmatrix} a_1 \\ \vdots\\ a_{|\mathcal{E}^E|} \end{pmatrix}, \hspace{3mm} a_l = \begin{pmatrix}0 \hdots \underbrace{1}_{i} \hdots \underbrace{-1}_{j} \hdots 0 \end{pmatrix}, \hspace{2mm} B = \operatorname{diag}(b_1,\ldots,b_{|\mathcal{E}^E|})
$$
Here, $M_a\in\mathbb{R}^{|\mathcal{E}^E|\times |\mathcal{V}^E|}$ refers to the adjacency matrix that reflects the network connectivity. The line $l$ connects the node $i$ to $j$, leading to the row vector $a_l \in \mathbb{R}^{|\mathcal{V}^E|}$ with non-zero entries $1$ and $-1$ at the $i^{\textrm{th}}$ and $j^{\textrm{th}}$ elements, respectively. 
The diagonal matrix $B\in\mathbb{R}^{|\mathcal{E}^E|\times |\mathcal{E}^E|}$ contains the susceptance (i.e., the inverse of the resistance) $b_l \in \mathbb{R}$ of line $l$. 

In the predominant portion of the literature~\cite{kannanoms13,GRIMM2016493}, power market operations are often confined to fixed time intervals in the sense that market clearing is considered once every 15-20 minutes. Such models are typically restricted to short-term planning and run for a time horizon of two to three days. This, however, does not naturally couple with gas networks, which are more dynamic. Continuous-time models have also been studied in the power market literature and their contributions towards policy design and model predictive control have become very important \cite{reeto08,terry16dvi}. While power markets have been studied from both Nash and Stackelberg~\cite{adler-twosettlement} (leader-follower) perspectives, the latter is applicable only when the ISO (regulator) is an agent. Transmission constraints have been addressed from both AC and DC angles. However, we note that AC load flow~\cite{fokken21} results in non-convex constraints (due to the phase angles). In contrast, DC load flow yields simpler linear constraints under certain practical assumptions on losses, allowing us  to prove the existence of equilibrium in Section~\ref{sec:existence-equilibria}.

\paragraph{Gas network constraints}
To keep the model for the gas network simple, we assume that the network $\mathcal{G}^H$ does not contain (active) compressor stations or valves, and it does not contain cycles. We note, however, that the former can be included by incorporating appropriate compressor models (see, e.g., \cite{MR4510217}, or associated engineering look-up tables \cite{compressor}) and by admitting integer controls \cite{MR3780468,hahn2017MixedIntegerPO}, cycles in graphs may lead to physically challenging situations which may complicate the solution; see, e.g., \cite{doi:10.1021/acs.iecr.0c04007}. We refer to \cite{data2040040} for example instances of (natural) gas networks.  

As the diameter of a pipe is much smaller than the length, a one-dimensional model for the gas flow in a pipe is suitable. Below,  we assume that all pipes in our network have the same diameter $d$ and cross-sectional area $A$.
%
%
For each pipe $e\in\mathcal{E}^H$ let $\ell_e>0$ denote its length, allowing us to associate $e$ with the interval $[0,\ell_e]$. The \textit{state} $y^e(x,t) = (p^e(x,t), q^e(x,t))$
of gas flow at each location $x \in (0,\ell_e)$ and each point in time $t \in (0,T)$ is given in terms of the pressure $p^e(x,t)$ (in the unit Pa=$N/m^2$) and mass flow rate $q^e(x,t)$ (in the unit kg/s). 
For the evolution of the state in a pipe $e$ we consider here the model
\begin{subequations}\label{ISO2-pipe-L-VR}
\begin{align}
p_t^e - \varepsilon p^e_{xx}+ \frac{\hat{c}^2}{A}q_x^e &= 0,\label{mh:par1}\\
q_t^e - \varepsilon q^e_{xx} +A p_x^e&=-\frac{\lambda_e \hat{c}^2 \tilde{q}^e}{2d A\tilde{p}^e}\left(2q^e- \frac{\tilde{q}^e}{\tilde{p}^e}p^e \right) - \frac{A \hat{g}\sin \alpha_e}{\hat{c}^2} p^e,\label{mh:par2}
\end{align}
\end{subequations}
where the parameters $\lambda_e$ and $\alpha_e$ denote the friction coefficient and the slope of the pipe, respectively, $\hat{c}$ is the speed of sound, and $\hat{g}$ the gravitational constant. This model relates to the ISO2 model of \cite{Domschke2017} for isothermal flow of (natural) gas with a constant compressibility factor and a linearized right-hand side around a given reference state $(\tilde{p}^e,\tilde{q}^e)$. 
The linearization is in particular suitable near an equilibrium state of the systems. One possibility to obtain suitable linearized terms is to utilize the iteration procedure of \cite{Grimm2021}. We also mention that the existence of solutions to our resulting GNEP is jeopardized by the nonlinear terms of ISO2 as these lead to a non-convex structure in infinite dimensions. In addition, the terms involving second derivatives are associated with a viscosity regularization for $0<\varepsilon (\ll 1)$ similar to \cite{Grimm2021}. 

The system described in \eqref{ISO2-pipe-L-VR} is complemented by initial, boundary, and coupling conditions. For the initial condition, it is assumed that at time $t=0$, an initial state $y(\cdot,0)$ is known. This translates to the \textit{initial conditions}
\begin{equation*}
\label{eq:IC}
p^e(x,0)=p^e_{0}(x)\quad \text{and} \quad q^e(x,0)=q^e_{0}(x)\quad \text{for a.a. }x\in (0,\ell_e),
\end{equation*}
with $p_0^e$ and $q_0^e$ given.
On boundary nodes $\mathcal{V}^H_\partial$, pressure and flow are determined by the \textit{boundary conditions} on $[0,T]$:
\begin{equation}\label{eq:BCs}
p^e(\nu,t) =  \sum_{i=1}^N \hat{p}_{i\nu}(t),\quad \text{and} \quad q^e(\nu,t)= \sum_{i=1}^N \hat{q}_{i\nu}(t) \quad \forall e \in k(\nu),
\end{equation}
with $\hat{p}_{i\nu}$ and $\hat{q}_{i\nu}$ given by agent $i$'s decisions for all $i$ and all $\nu \in \mathcal{V}_\partial^H$; recall in particular 
\eqref{eq:conversion-conditions}. We use the notation
\begin{equation*}
    k(\nu) \coloneqq \{e \in \mathcal{E}^H \mid n^e(\nu) \neq 0 \},
\end{equation*}
with
\begin{equation*}
    n^e(\nu) = \begin{cases} -1, \quad & \text{if } \nu =0,\\
\phantom{-}1, &\text{if } \nu = \ell_e,\\
    \phantom{-}0, &\text{else}.
    \end{cases}
\end{equation*}
For simplicity, we assume that $|k(\nu)| = 1$ for boundary nodes.

Here, and throughout the rest of the paper, we assume that the initial conditions are compatible with the given boundary and coupling conditions at time $t=0$. In particular, this requires that
\begin{equation*}
\label{eq:compatibility-conditions}
    p^e_0(\nu) = p^e(\nu,0) \quad \text{and} \quad q^e_0(\nu) = q^e(\nu,0) \quad \forall \nu \in \mathcal{V}^H.
\end{equation*}

The boundary conditions in \eqref{eq:BCs} constitute the only way agents can control gas flow on the network. 
\noindent For inner nodes $\nu \in \mathcal{V}^H_0$, which are not controlled by agents but are rather implicitly defined, we have the following \textit{coupling conditions} on $[0,T]$:
\begin{subequations}
\label{eq:CC-explicit}
\begin{align}
& p^e(\nu,t) = p^{e'}(\nu,t) \quad \forall e,e' \in k(\nu), \label{eq:CC-explicit-1}\\
&\sum_{e \in k(\nu)} q^e(\nu,t) n^e(\nu)=0, \label{eq:CC-explicit-2}\\
& \sum_{e \in k(\nu)}p_x^e(\nu,t) n^e(\nu)= 0, \label{eq:CC-explicit-3}\\
& {q}_x^e(\nu,t) ={q}_x^{e'}(\nu,t) \quad \forall e, e' \in k(\nu). \label{eq:CC-explicit-4}
\end{align} 
\end{subequations}

\begin{remark}
The conditions \eqref{eq:CC-explicit-1}--\eqref{eq:CC-explicit-2} are consistent with the limiting (with $\varepsilon =0$) ISO2 model from \cite{Domschke2017}.
We note that it is possible to emphasize the dependence of the viscosity parameter $\varepsilon$ in \eqref{eq:CC-explicit-3}--\eqref{eq:CC-explicit-4} by including zero-order terms there. This would come from integrating the first-order terms by parts in the expression \eqref{eq:integration-by-parts} introduced later and would result in the following conditions:
\begin{equation}
\label{eq:alternative-CCs}
\begin{aligned}
\sum_{e \in k(\nu)} \left(-\varepsilon p_x^e(\nu,t) + \frac{\hat{c}^2}{A}q^e(\nu,t)\right) n^e(\nu)= 0, \\
-\varepsilon {q}_x^e(\nu,t) + Ap^e(\nu,t)= -\varepsilon {q}_x^{e'}(\nu,t) +Ap^{e'}(\nu,t)\quad \forall e, e' \in k(\nu).
\end{aligned}
\end{equation}
This kind of coupling conditions was introduced in the paper \cite{Egger2021}, where a viscosity-regularized transport equation on networks was investigated. Of course, these conditions are consistent with \eqref{eq:CC-explicit} since by plugging \eqref{eq:CC-explicit-1}--\eqref{eq:CC-explicit-2} into \eqref{eq:alternative-CCs}, we recover \eqref{eq:CC-explicit-3}--\eqref{eq:CC-explicit-4}.
The conditions \eqref{eq:alternative-CCs} and hence \eqref{eq:CC-explicit} are in principle compatible as $\varepsilon \rightarrow 0$ with the ISO2 conditions \eqref{eq:CC-explicit-1}--\eqref{eq:CC-explicit-2} (coming from the purely hyperbolic case) if $p_x^e$ and $q_x^e$ are bounded on all edges.
While out of the scope of this paper, an analysis of the viscosity limit as $\varepsilon$ tends to zero in the style of \cite{Egger2021} will be subject of future work.
\end{remark}


The pressure and mass flow rate must remain within certain operational thresholds for every point in the network and at all times. Hence, we require on every edge $e\in\mathcal{E}^H$ that 
\begin{equation}
\label{eq:state-constraint}
    0 \leq q^e(x,t) \leq q_e^{\max}, \quad 0 \leq p^e(x,t) \leq p_e^{\max}\quad\text{for a.a. }(x,t)\in (0,\ell_e) \times [0,T]
\end{equation}
for pre-specified bounds $q_e^{\max}>0$ and $p_e^{\max}>0$. Note that the constraint \eqref{eq:state-constraint} has to be consistent with the constraint on $\hat{p}_{i\nu}$ from the agent's admissible set \eqref{eq:Uad}. We will see later in Lemma \ref{lemma:regularity-PDE} that on each pipe, pressure and flow are continuous with respect to $x$ up until the boundary. In that sense, $p_e^{\max}$ should coincide with $p_\nu^{\max}$ for $e\in k(\nu)$.



\begin{remark}
In our simplified model, the storage of gas is done within the pipeline itself; still leading to the interesting aspect of line packing, partly also contributing to the robustness when running such a network.
It is possible to allow for storage at nodes in the form of fuel cells such as the PEMC (Proton Exchange Membrane Cell). 
More details on fuel cells can be found in~\cite{osti_book88,blomen2013fuel} and details on costs can be found in~\cite{shahid17}. Storage can be modeled using the following adjustment (using the example of power-to-gas nodes). Let us assume that agent $i$ has its own individual storage facility at node $\nu$. 
Given an initial amount $Q_{i\nu}(t_0)$ of gas at $\nu$ at time $t_0$, the amount of gas in storage at a later time $t$ is defined by the following condition:
\begin{equation*}
\label{eq:Ptg-constraint}
Q_{i\nu}(t)  = \underbrace{Q_{i\nu}(t_0)}_{\text{kg at time $t_0$}} - \underbrace{\int_{t_0}^{t} \hat{q}_{i\nu}(\tau) \D \tau}_{\text{kg removed since time $t_0$}} + \underbrace{\int_{t_0}^{t} \eta_{\nu} c_{i\nu}(\tau) \D \tau.}_{\text{kg added using power}}
\end{equation*}
Naturally, in any storage facility, there cannot be a negative storage amount at any given time, and there is also a maximum capacity of storage $Q_{i\nu}^{\max}$, leading to additional storage constraints of the form $0 \leq Q_{i\nu}(t)\leq Q_{i\nu}^{\max}.$ Storage could be modeled analogously at GtP nodes and hydrogen sale nodes. To keep the model simple, we focus on storage in the form of line packing, only.
\end{remark}

\subsection{The overall model}
\label{sec:summary-model}
Summarizing, the full model proposed here is a GNEP where, given the decisions $u_{-i}$ of all other agents, the problem of agent $i$ reads
\begin{subequations}
\label{eq:gas-model-detailed}
 \begin{align}
 \label{eq:gas-model-detailed-obj}
\max_{u_i \in U_i}  \int_0^T  & \left[\sum_{\nu \in \mathcal{V}^s}   s_{i\nu}(\tau) \left(\bar{a}_\nu(\tau) - \bar{b}_\nu(\tau) \sum_{j=1}^N s_{j\nu}(\tau)\right) \right. \nonumber  \\
 &\left. \quad -\sum_{\nu' \in \mathcal{V}^g} g_{i\nu'}(\tau) \kappa_{i\nu'}     -\sum_{\nu'' \in \mathcal{V}^{\textup{PtG}}}  c_{i\nu''}(\tau)  \kappa_{i\nu''} \right]\D \tau,
 \end{align}
subject to electric grid constraints
\begin{alignat}{2}
\sum_{\nu \in \mathcal{V}^E} (s_{i\nu}(t) + c_{i\nu}(t)-g_{i\nu}(t)) &= 0 \quad  \hspace{1mm} &&\text{in } (0,T) \,\forall i \in \left\{1,\hdots,N \right\},\label{eq:electricity-model-detailed-1}\\
 \sum_{\nu \in \mathcal{V}^E} \hat{A}_{e\nu} \sum_{i=1}^N (s_{i\nu}(t) + c_{i\nu}(t)-g_{i\nu}(t)) &\leq \theta_e &&\text{in } (0,T) \,\forall e \in \mathcal{E}^E, \, \hspace{1mm} \label{eq:electricity-model-detailed-2}
\end{alignat}
hydrogen pipeline constraints 
\begin{alignat}{2}
&  p_t^e - \varepsilon p^e_{xx}+ \frac{\hat{c}^2}{A}q_x^e = 0   &&\text{in } e\times (0,T)\, \forall e \in \mathcal{E}^H, \label{eq:gas-model-detailed-gas-1}\\ 
& q_t^e - \varepsilon q^e_{xx} +A p_x^e=-\frac{\lambda_e \hat{c}^2 \tilde{q}^e}{2d A\tilde{p}^e}\left(2q^e - \frac{\tilde{q}^e}{\tilde{p}^e}p^e \right)-&& 
\frac{A \hat{g}\sin \alpha_e}{\hat{c}^2} p^e  \nonumber\\
& &&
\text{in } e\times (0,T)\, \forall e \in \mathcal{E}^H,\label{eq:gas-model-detailed-gas-2}\\
& p^e(x,0) = p^e_{0}(x), \quad q^e(x,0)=q^e_{0}(x) && \text{in } e\, \forall e \in \mathcal{E}^H, \label{eq:gas-model-detailed-gas-3}\\
& p^e(\nu,t) = \sum_{i=1}^N \hat{p}_{i\nu}(t), \quad q^e(\nu,t) = \sum_{i=1}^N\hat{q}_{i\nu}(t) &&\text{in } [0,T] \, \forall \nu \in  \mathcal{V}^H_{\partial} \,\forall e \in k(\nu),\label{eq:gas-model-detailed-gas-4}\\
& p^{e}(\nu,t) = p^{e'}(\nu,t) &&\text{in } [0,T]\,\forall \nu \in \mathcal{V}^H_{0} \, \forall e, e' \in k(\nu),\label{eq:gas-model-detailed-gas-5}\\
&  \sum_{e \in k(\nu)}q^e(\nu,t) n^e(\nu)= 0 &&\text{in } [0,T] \, \forall \nu \in  \mathcal{V}^H_{0},  \label{eq:gas-model-detailed-gas-6} \\
& \sum_{e \in k(\nu)} p_x^e(\nu,t) n^e(\nu)= 0 &&\text{in } [0,T] \, \forall \nu \in  \mathcal{V}^H_{0},  \label{eq:gas-model-detailed-gas-7} \\
&  {q}_x^e(\nu,t) = {q}_x^{e'}(\nu,t) \quad &&\text{in } [0,T] \,\forall \nu \in \mathcal{V}^H_{0} \, \forall e, e' \in k(\nu),\label{eq:gas-model-detailed-gas-8}\\
& 0 \leq q^e(x,t) \leq q_e^{\max}, \quad 0 \leq p^e(x,t) \leq p_e^{\max} \quad && \text{in } e\times [0,T] \, \forall e \in \mathcal{E}^H,\label{eq:gas-model-detailed-gas-9}
\end{alignat}
 electricity/hydrogen conversion constraints 
\begin{alignat}{2}
\hat{q}_{i\nu}(t) &= \eta_{\nu} c_{i\nu}(t) \qquad \qquad &&\text{in } (0,T) \, \forall \nu \in \mathcal{V}^{\textup{PtG}} \,\forall i \in\{1,\dots, N\},\label{eq:conversion-model-detailed-1}\\
\hat{q}_{i\nu}(t) &= \eta_{\nu} g_{i\nu}(t) &&\text{in }  (0,T) \, \forall \nu \in \mathcal{V}^{\textup{GtP}} \, \forall i \in\{1,\dots, N\}, \label{eq:conversion-model-detailed-2}\\
\hat{q}_{i\nu}(t) &= \eta_{\nu}  s_{i\nu}(t) &&\text{in } (0,T) \, \forall \nu \in \mathcal{V}^s\cap\mathcal{V}^H\, \forall i \in\{1,\dots, N\}, \label{eq:conversion-model-detailed-3}
\end{alignat}
and control constraints
\begin{alignat}{2}
&0 \leq g_{i\nu}(t)\leq g_\nu^{\max} \quad\quad &&\text{in } [0,T]\,\forall \nu \in \mathcal{V}^g,  \label{eq:control-model-detailed-1} \\
    &0 \leq s_{i\nu}(t) \leq s_\nu^{\max}\,  && \text{in } [0,T]\,\forall \nu \in \mathcal{V}^s, \label{eq:control-model-detailed-2}\\
    & 0 \leq c_{i\nu}(t) \leq c_\nu^{\max} &&\text{in } [0,T]\,\forall \nu \in \mathcal{V}^{\textup{PtG}},\label{eq:control-model-detailed-3}\\
    & 0 \leq \hat{p}_{i\nu}(t) \leq p_\nu^{\max} &&\text{in } [0,T]\,\forall \nu \in \mathcal{V}^H_\partial, \label{eq:control-model-detailed-4}\\
     &\int_0^T \Big( \sum_{\nu \in \mathcal{V}^{\textup{PtG}}} \hat{q}_{i\nu}(\tau) - \sum_{\nu' \in \mathcal{V}^{\textup{GtP}}} \hat{q}_{i\nu'}(\tau) - \sum_{\nu'' \in \mathcal{V}^{s}\cap \mathcal{V}^H} &&\hat{q}_{i\nu''}(\tau) \Big)\D \tau  \geq 0. \label{eq:control-model-detailed-5}
\end{alignat}
\end{subequations}

\section{Existence of an equilibrium}
\label{sec:existence}
In this section, we investigate the existence of equilibria to the GNEP given by the agents' problems \eqref{eq:gas-model-detailed}. To this end, we present in Section \ref{sec:assumptions} our analytic framework and technical assumptions. The existence of solutions to the PDE system is given in Section \ref{sec:existence-PDE}.

\subsection{Assumptions}
\label{sec:assumptions}
Before we introduce the mathematical assumptions needed to establish existence of an equilibrium, we provide some notation for the function spaces used for the gas constraints. 
Given a Banach space $X$, the norm and dual space are denoted by $\lVert \cdot \rVert_X$ and $X^*$, respectively. We shall use the shorthand $L^2(e) := L^2(0,\ell_e)$ for the Lebesgue space of square integrable functions on $(0,\ell_e)$ and the product space
\[
 H\coloneqq \Big(\prod_{e \in \mathcal{E}^H} L^2(e)\Big)^2 = \{ (p,q) \mid p^e \in L^2(e), q^e \in L^2(e) \, \forall e \in \mathcal{E}^H \}.
 \]
For test functions, we define the following coupling conditions at inner nodes 
\begin{equation} 
\label{eq:test-function-conditions}
\begin{aligned}
&\quad  p^e(\nu) = p^{e'}(\nu) && \forall \nu \in  \mathcal{V}^H_{0} \, \forall e,e' \in k(\nu),\\
&\sum_{e \in k(\nu)} q^e(\nu) n^e(\nu)=0 &&\forall \nu \in  \mathcal{V}^H_{0},
\end{aligned}
\end{equation}
so that the test function space is defined by
\[
V  \coloneqq \Big\lbrace (p,q) \in \Big(\prod_{e \in \mathcal{E}^H} H^1(e)\Big)^2 \mid (p,q) \text{ satisfy \eqref{eq:test-function-conditions} in the sense of traces} \Big\rbrace.
\]
For the analysis in the following section, we also use the following space with disappearing boundary
\begin{equation}
\label{eq:V0-definition}
\begin{aligned}
V_0 \coloneqq \Big\lbrace (p,q) \in V \mid & \, v \text{ satisfies } v(\nu,\cdot)=0 \, \forall \nu \in \mathcal{V}_\partial^H \text{ in the sense of traces}\Big\rbrace.
\end{aligned}
\end{equation}
The separable Hilbert space $H$ is equipped with the inner product 
\begin{equation}
\label{eq:scalar-product1}
    (v,w)_{H} = \sum_{e \in \mathcal{E}^H} (v^{e,p},w^{e,p})_{L^2(e)} + (v^{e,q},w^{e,q})_{L^2(e)} 
\end{equation}
and the separable Hilbert spaces $V, V_0$ are equipped with the inner product
\begin{equation}
\label{eq:scalar-product2}
\begin{aligned}
    (v,w)_{V} = (\nabla v,\nabla w)_{H} + (v,w)_{H};
    \end{aligned}
\end{equation}
see Section \ref{sec:function-spaces}. For a Banach space $X$, we consider the vector-valued mappings
\[
t \in [0,T] \mapsto z(t)=(z^p(t),z^q(t)) \in X.
\]
Provided $X$ is a separable Hilbert space, then for $t>0$, the space
\begin{align*}
L^2(0,t;X) = \Big\lbrace z\colon [0,t] &\rightarrow X \textup{ strongly measurable } \mid \\
&\lVert z\rVert_{L^2(0,t;X)}:=\Big(\int_0^t \lVert z(\tau) \rVert_X^2\D \tau \Big)^{1/2} <\infty \Big\rbrace
\end{align*}
is a Hilbert space. For the solution space, we define 
\[    W(0,t;H,V) \coloneqq \{v \in L^2(0,t;V)\mid \partial_t v \in L^2(0,t; V^*)\}
\]
equipped with the norm
\[
\lVert z \rVert_{W(0,T;H,V)} := \sqrt{\lVert z \rVert_{L^2(0,T;V)}^2 +  \lVert z_t \rVert_{L^2(0,T;V^*)}^2}.
\]

We will now present the overarching assumptions made in our model.

\begin{assumption}
\label{ass:SGNEP}
\subasu (Agents' objectives) We assume that  $\bar{a}_{\nu} \in L^\infty(0,T)$ and $\bar{b}_{\nu} \in L^\infty(0,T)$ are nonnegative a.e.~in $(0,T)$ and for all $\nu \in \mathcal{V}^s$. The costs $\kappa_{i\nu}$ are nonnegative for all $i$ and $\nu \in \mathcal{V}^g \cup \mathcal{V}^{\textup{PtG}}$.\\
\subasu (Private restrictions) The efficiency $\eta_\nu$ is positive for $\nu \in \mathcal{V}^H_\partial$. So are the bounds $g_{\nu}^{\max}$ ($\nu \in \mathcal{V}^H_{g}$), $s_{\nu}^{\max}$ ($\nu \in \mathcal{V}^H_s$), $c_{\nu}^{\max}$ ($\nu \in \mathcal{V}^{PtG}$), and $p_{\nu}^{\max}$ ($\nu \in \mathcal{V}^H_\partial$).\label{ass:private-restrictions}\\
\subasu (Power grid) The parameters $\theta_e > 0$ and $\hat{A}_{e\nu}$ are bounded for every $e \in \mathcal{E}^E$ and $\nu \in \mathcal{V}^E$. 
\label{ass:operational-grid}\\
\subasu (Gas network) The parameters $\varepsilon, \hat{c},A,d,\alpha_e,\lambda_e$ are  positive for all $e \in \mathcal{E}^H$. The reference points $\tilde{p}^e,\tilde{q}^e$ belong to $C([0,\ell_e]\times[0,T])$ and $\tilde{p}^e \neq 0$ a.e.~in $e \times (0,T)$. The initial conditions satisfy $y_0:=(p_0,q_0) \in H$.\label{ass:PDE}\\
\subasu (State constraints) We have $p^{\max}_e>0$ and $ q^{\max}_e>0$ for every $e \in \mathcal{E}^H$.\label{ass:operational-pipe}
\end{assumption}

\subsection{Existence of solutions to PDE}
\label{sec:existence-PDE}
Now we prove that the PDE is uniquely solvable for any combination of decisions made by the agents at the boundary nodes. Moreover, the continuous dependence on the agents' collective decisions at the boundary nodes $\mathcal{V}_\partial^H$ is shown. 
\begin{lemma}
\label{lemma:regularity-PDE} 
Suppose Assumption \ref{ass:PDE} is satisfied. Then, for every $(\mathfrak{p},\mathfrak{q}) \in H^1(0,T)^{n_h}\times H^1(0,T)^{n_h}$ compatible with the initial condition $y_0 = (p_0,q_0)$, there exists a unique weak solution $y=(p,q) \in W(0,T;H,V)$ satisfying \eqref{eq:gas-model-detailed-gas-1}--\eqref{eq:gas-model-detailed-gas-8}. Moreover, there exists a constant $C>0$ such that the following a priori estimate holds for all $t\in (0,T]$:
\begin{equation}
\label{eq:a-priori-estimate}
\lVert y\rVert_{L^2(0,t;V)}^2 + \lVert  y_t\rVert_{L^2(0,t;V^*)}^2 \leq 
C \left( \|y_0\|_H^2 + \sum_{\nu \in \mathcal{V}_\partial^H} (\lVert \mathfrak{p}^{\nu}\rVert^2_{H^1(0,t)} +\lVert \mathfrak{q}^{\nu}\rVert^2_{H^1(0,t)})  \right).
\end{equation}
\end{lemma}

The regularity obtained for the state $y$ in Lemma \ref{lemma:regularity-PDE} will be sufficient to show existence of equilibria; we remark that higher regularity would be needed to establish the existence of Lagrange multipliers as was done in \cite{Grimm2021}. 
The proof of Lemma \ref{lemma:regularity-PDE} proceeds along the lines of \cite{Grimm2021}, with the additional complication that we are considering a network, instead of the case of a single pipe. Before getting to the proof, we need some preliminary results.

Let
      $\theta_1 \coloneqq \hat{c}^2/A$ 
      and $\theta_2 \coloneqq A$, and define
for each point $(x,t)$ 
\begin{equation*}
    \gamma_1^e(x,t) \coloneqq \frac{\lambda_e \hat{c}^2 \tilde{q}^e(x,t)}{d A \tilde p^e(x,t)} \quad \text{and} \quad \gamma_2^e(x,t) \coloneqq \frac{A \hat{g} \sin \alpha_e}{\hat{c}^2} -\frac{\lambda_e \hat{c}^2  (\tilde q^e(x,t))^2}{2 d A  (\tilde p^e(x,t))^2},
\end{equation*}
Then the full system of equations  \eqref{eq:gas-model-detailed-gas-1}--\eqref{eq:gas-model-detailed-gas-8} in new notation reads:
\begin{subequations}
\label{eq:gas-model-detailed'}
\begin{alignat}{2}
& p_t^e - \varepsilon p_{xx}^e+ \theta_1 q_x^e = 0   &&\text{a.e.~in } (0,\ell_e) \times (0,T)\, \forall e \in \mathcal{E}^H, \label{eq:gas-model-detailed-1'}\\ 
& q_t^e - \varepsilon q^e_{xx} +\theta_2 p_x^e +\gamma_1^e q^e + \gamma_2^e p^e =0 \quad && 
\text{a.e.~in } (0,\ell_e)\times (0,T)\, \forall e \in \mathcal{E}^H,\label{eq:gas-model-detailed-2'}\\
& p^e(x,0) = p^e_{0}(x), \quad q^e(x,0) = q^e_{0}(x) && \text{a.e.~in } (0,\ell_e)\, \forall e \in \mathcal{E}^H, \label{eq:gas-model-detailed-3'}\\
& p^e(\nu,t) = \mathfrak{p}^{\nu}(t), \quad q^e(\nu,t) = \mathfrak{q}^{\nu}(t) &&\text{a.e.~in } [0,T] \, \forall \nu \in  \mathcal{V}^H_{\partial} \,\forall e \in k(\nu),\label{eq:gas-model-detailed-4'}\\
& \sum_{e \in k(\nu)}q^e(\nu,t) n^e(\nu)= 0 &&\text{a.e.~in } [0,T] \, \forall \nu \in  \mathcal{V}^H_{0}, \label{eq:gas-model-detailed-5'} \\
& p^e(\nu,t)=p^{e'}(\nu,t)   
 &&\text{a.e.~in } [0,T] \,\forall \nu \in \mathcal{V}^H_{0} \, \forall e, e' \in k(\nu),\label{eq:gas-model-detailed-6'}\\
& \sum_{e \in k(\nu)}  p_x^e(\nu,t)  n^e(\nu)= 0 &&\text{a.e.~in } [0,T] \, \forall \nu \in  \mathcal{V}^H_{0}, \label{eq:gas-model-detailed-7'} \\
&  {q}_x^e(\nu,t) = {q}_x^{e'}(\nu,t)  &&\text{a.e.~in } [0,T] \,\forall \nu \in \mathcal{V}^H_{0} \, \forall e, e' \in k(\nu).\label{eq:gas-model-detailed-8'}
 \end{alignat}
\end{subequations}
A complication in showing existence of solutions in \eqref{eq:gas-model-detailed'} is the time dependence of the functions $\gamma_1$ and $\gamma_2$ coming from our linearization of the semilinear term, which prevent us from easily applying semigroup theory. For this reason, we rely on existence using \cite[Section 3, Theorem 1.2]{Lions1971} based on the Galerkin method. Additionally, the control variables $(\mathfrak{p}^{\nu},\mathfrak{q}^{\nu})$ on the boundary nodes $\mathcal{V}_\partial^H$ are of Dirichlet type. In combination with the parabolic problem, this means that $L^2$ regularity with respect to $t$ is not sufficient, unless we are content with very weak solutions for the state $(p,q)$. In view of the additional state constraints, it will be beneficial to require higher regularity of the control. The intricacies of Dirichlet boundary controls in the parabolic setting are explored in \cite[Section 9]{Lions1971}; see also \cite{kunisch_constrained_2007}.

\paragraph{Variational formulation.}
Before we define our weak solution, we motivate its definition. We start by fixing $t$, multiplying \eqref{eq:gas-model-detailed-1'}-\eqref{eq:gas-model-detailed-2'} by sufficiently smooth test functions $\phi^{e,p},\phi^{e,q}$, and adding the resulting equations. Then we obtain 
\begin{align*}
    &\sum_{e\in \mathcal{E}^H} \int_e   ( p_t^{e} - \varepsilon  p_{xx}^{e} + \theta_1  q_x^e )\phi^{e,p} + (q_t^e - \varepsilon  q_{xx}^e + \theta_2  p_x^e + \gamma_1^e q^e + \gamma_2^e p^e)\phi^{e,q}\D x= 0.
\end{align*}
Integration by parts for the  terms with second-order (spatial) derivatives yields
\begin{equation}
\label{eq:integration-by-parts}
\begin{aligned}
    &\sum_{e\in \mathcal{E}^H} \int_e   - \varepsilon  p_{xx}^{e}- \varepsilon  q_{xx}^e \phi^{e,q}\D x\\
    & \quad =  \varepsilon \sum_{e\in \mathcal{E}^H} \int_e    p_{x}^{e} \phi_x^{e,p} + q_{x}^e \phi_x^{e,q}\D x\\
    &\quad \quad-\varepsilon\sum_{\nu \in \mathcal{V}^H} \sum_{e\in k(\nu)} \Big( p_x^e(\nu)  \phi^{e,p}(\nu)n^e(\nu)+  q_x^e(\nu) \phi^{e,q}(\nu)n^e(\nu)\Big).
   \end{aligned}
   \end{equation}
Continuing by distinguishing between boundary and inner nodes, we arrive at
\begin{align*}
&\sum_{\nu \in \mathcal{V}^H} \sum_{e\in k(\nu)} \Big( p_x^e(\nu)  \phi^{e,p}(\nu)n^e(\nu)+  q_x^e(\nu) \phi^{e,q}(\nu)n^e(\nu) \Big)\\
& \quad = \sum_{\nu \in \mathcal{V}_0^H} \hat{\varphi}_1^\nu \sum_{e\in k(\nu)}  p_x^e(\nu) n^e(\nu)+ \sum_{\nu \in \mathcal{V}_0^H} \hat{\varphi}_2^{\nu} \sum_{e\in k(\nu)} \phi^{e,q}(\nu)n^e(\nu) \\
&  \quad \quad + \sum_{\nu \in \mathcal{V}_\partial^H} \mathfrak{p}^{\nu} \sum_{e\in k(\nu)} p_x^e(\nu)  n^e(\nu)+ \sum_{\nu \in \mathcal{V}_\partial^H} \mathfrak{q}^{\nu} \sum_{e\in k(\nu)}   q_x^e(\nu) n^e(\nu) \\
    & \quad =  \sum_{\nu \in \mathcal{V}_\partial^H} \mathfrak{p}^{\nu} \sum_{e\in k(\nu)}  p_x^e(\nu) n^e(\nu)+ \sum_{\nu \in \mathcal{V}_\partial^H} \mathfrak{q}^{\nu} \sum_{e\in k(\nu)}   q_x^e(\nu)  n^e(\nu),
\end{align*}
where 
$\hat{\varphi}_1^{\nu}:=\phi^{e,p}(\nu)$ and
 $\hat{\varphi}_2^\nu:= q_x^{e}(\nu)$
for $(p,q), (\phi^p,\phi^q) \in V$.
Note that the last equality above follows from \eqref{eq:gas-model-detailed-8'} and \eqref{eq:gas-model-detailed-6'}. 
These computations motivate the following bilinear form for $z,\phi \in V$:
\begin{align}
   a(z,\phi;t) \coloneqq \sum_{e \in \mathcal{E}^H} \int_{e} &\left\lbrace \varepsilon  z_x^{e,p} \phi_x^{e,p} + \theta_1 z_x^{e,q} \phi^{e,p} +\varepsilon z_x^{e,q} \phi_x^{e,q} + \theta_2  z_x^{e,p} \phi^{e,q} \right. \notag \\
   & \quad \left. + \gamma_1^e(\cdot,t) z^{e,p} \phi^{e,q} + \gamma_2^e(\cdot,t) z^{e,q} \phi^{e,q} \right\rbrace \D x.\label{eq:a_bilin_def}
\end{align}
In order to handle the inhomogenous boundary conditions, let $w\in W(0,T;H,V)$ be the piecewise affine linear function satisfying $w(\nu,t)=(\mathfrak{p}^\nu(t),\mathfrak{q}^\nu(t))$ for all $\nu \in \mathcal{V}_\partial^H$ and $w(\nu,t)=(0,0)$ for all $\nu \in \mathcal{V}_0^H$ (in the sense of traces). Then the weak formulation for \eqref{eq:gas-model-detailed'} is given by: find $y=(p,q) \in W(0,T;H,V)$ such that $y-w \in W(0,T;H,V_0)$ and
\begin{equation}\label{eq:varform}
    \int_0^T \left \langle y_t(t), \phi(t) \right \rangle_{V^*, V} + a(y(t),\phi(t);t) \D t = 0\quad \forall \phi \in W(0,T;H,V),
\end{equation}
with the initial condition  \eqref{eq:gas-model-detailed-3'}. For the definition of $V_0$ we refer to \eqref{eq:V0-definition}.

\paragraph{Recovery of classical solution.}
We will now check that our choice of test spaces is correct. First, notice that the condition $y-w \in W(0,T;H,V_0)$ from the weak formulation implies \eqref{eq:gas-model-detailed-4'}.
Let us fix an arbitrary $t$ and suppose that $p,q \in C^2(\prod_{e \in \mathcal{E}^H}[0,\ell_e])$ is a classical solution to \eqref{eq:gas-model-detailed'}. From the integrand appearing in  \eqref{eq:varform}, we have after integrating by parts
\begin{align}
 0 &= \left \langle y_t, \phi \right \rangle_{V^*, V} + a(y,\phi;t)   \nonumber \\
 & = \sum_{e\in \mathcal{E}^H}\left\lbrace ( p_t^{e} - \varepsilon  p_{xx}^{e} + \theta_1  q_x^e ,\phi^{e,p})_{L^2(e)} \right. \nonumber \\
 & \hspace{2cm} \left. + (q_t^e - \varepsilon  q_{xx}^e + \theta_2  p_x^e + \gamma_1^e q^e + \gamma_2^e p^e,\phi^{e,q})_{L^2(e)} \right\rbrace  \label{eq:classical-check-2}\\
 &  \quad + \sum_{\nu \in \mathcal{V}^H} \sum_{e\in k(\nu)} \Big( \varepsilon p_x^e(\nu,t)  \phi^{e,p}(\nu)n^e(\nu)  + \varepsilon q_x^e(\nu,t)  \phi^{e,q}(\nu)n^e(\nu)\Big)\label{eq:classical-check-3}
\end{align}
for all $\phi^{p},\phi^{q} \in C^1(\prod_{e \in \mathcal{E}^H}[0,\ell_e])$ satisfying the boundary conditions \eqref{eq:gas-model-detailed-5'} and \eqref{eq:gas-model-detailed-6'}, respectively. The above equality must be satisfied, in particular, for all functions $\phi^p, \phi^q$ that vanish on the boundary of each edge (implying that the terms in \eqref{eq:classical-check-3} vanish). We can deduce that the sum in \eqref{eq:classical-check-2} is equal to zero and thus \eqref{eq:gas-model-detailed-1'}--\eqref{eq:gas-model-detailed-2'} is true for fixed $t$. This further implies that the sum over all interior nodes in the expression in \eqref{eq:classical-check-3} is equal to zero, or equivalently,
\begin{equation}
\begin{aligned}
\label{eq:boundary-conditions-separately}
\sum_{\nu \in \mathcal{V}_0^H} \sum_{e\in k(\nu)}  \varepsilon p_x^e(\nu,t) \phi^{e,p}(\nu)n^e(\nu) &=0, \\ 
\sum_{\nu \in \mathcal{V}_0^H} \sum_{e\in k(\nu)}
\varepsilon q_x^e(\nu,t)  \phi^{e,q}(\nu)n^e(\nu)&=0.
\end{aligned}
\end{equation}
Since $\phi^{p}$ satisfies  \eqref{eq:gas-model-detailed-6'}, we have that 
\begin{align*}
&\sum_{\nu \in \mathcal{V}_0^H} \sum_{e\in k(\nu)} \varepsilon p_x^e(\nu,t)  \phi^{e,p}(\nu)n^e(\nu) =  \sum_{\nu \in \mathcal{V}_0^H} \hat{\varphi}_1^{\nu}\sum_{e\in k(\nu)} \varepsilon p_x^e(\nu,t) n^e(\nu)
\end{align*}
for arbitrary constants $\hat{\varphi}_1^{\nu}$ on each node $\nu$. With that, we recover the conditions \eqref{eq:gas-model-detailed-7'}. 
Moreover, since $\phi^{q}$ satisfies \eqref{eq:test-function-conditions}, we see that \eqref{eq:gas-model-detailed-8'} must also be true, since
the second expression in \eqref{eq:boundary-conditions-separately} is of the form $\sum_{n=1}^N a_n b_n = 0$ with $\sum_{n=1}^N b_n = 0$.

\paragraph{Properties of the bilinear form.}
To show that a weak solution exists, we use the following result.
\begin{lemma}\label{lem:propertiesbilinearform} Suppose Assumption \ref{ass:PDE} is satisfied. Then, $t \mapsto a(z,\phi;t)$ (see \eqref{eq:a_bilin_def}) is measurable for all $z,\phi \in V_0$. Furthermore, the bilinear form $a(\cdot,\cdot;t) \colon V_0\times V_0 \rightarrow \R$ is continuous, i.e., there exists a constant $K>0$ independent of $t$ such that
    \begin{equation*}
        |a(z,\phi;t)| \leq K \|z\|_{V} \|\phi\|_{V} \quad \forall z,\phi \in V_0 \textup{ a.e.~in } (0,T);
    \end{equation*}
and it is weakly coercive, i.e., there exists constants $\beta \geq 0$ and $\alpha_e>0$ such that
    \begin{equation*}
        a(z,z;t) + \beta \|z\|_{H}^2 \geq \alpha_e \|z\|^2_{V} \quad \forall z \in V_0 \textup{ a.e.~in } (0,T).
    \end{equation*}
\end{lemma}

\begin{proof} 
By assumption, $\tilde{p}^e,\tilde{q}^e$ belongs to $C([0,\ell_e]\times[0,T])$, rendering $(x,t) \mapsto (\tilde{p}^e,\tilde{q}^e)(x,t)$ measurable, from which we obtain measurability of the functions $ t \mapsto \gamma_1^e(x,t)$ and $t \mapsto \gamma_2^e(x,t)$ for all $x$. From this, it is straightforward to argue the measurability of the bilinear form, for instance by expressing the integral as the limit of measurable maps.

Now, let $t \in (0,T)$ be arbitrary but fixed. For continuity, we note that Assumption \ref{ass:PDE} ensures that $\bar{\gamma}_1:=\sup_{e \in \mathcal{E}^H}\lVert \gamma_1^e \rVert_{L^\infty(e\times(0,T))}$ and $\bar{\gamma}_2:=\sup_{e \in \mathcal{E}^H}\lVert \gamma_2^e \rVert_{L^\infty(e\times(0,T))}$ are finite, allowing us to bound terms from above in the usual way.  For weak coercivity, we will repeatedly use Cauchy--Schwarz and Young's inequalities for the inner product of functions $a$ and $b$ and an arbitrary positive constant $\rho^e$:
\[
(a,b)_{L^2(e)} \geq - \lVert a\rVert_{L^2(e)} \lVert b\rVert_{L^2(e)} \geq -\frac{\rho^e}{2}\lVert a\rVert_{L^2(e)}^2 - \frac{1}{2\rho^e} \lVert b\rVert_{L^2(e)}^2.
\]
Using these estimates, we obtain positive constants $\rho_1^e$, $\rho_2^e$, and $\rho_3^e$ so that
\begin{align*}
 a(z,z;t) &\geq \sum_{e \in \mathcal{E}^H}\lVert z_x^{e,p} \rVert_{L^2(e)}^2 \left( \varepsilon - \frac{\theta_2\rho_2^e}{2 }\right)+ \lVert z_x^{e,q} \rVert_{L^2(e)}^2 \left( \varepsilon-\frac{\theta_1\rho_1^e}{2}\right)  \\
 &\quad \quad  +\sum_{e \in \mathcal{E}^H}- \lVert z^{e,p}\rVert_{L^2(e)}^2 \left( \frac{\theta_1}{2 \rho_1^e}+\frac{\bar{\gamma}_1 \rho_3^e}{2}\right) -\lVert z^{e,q}\rVert_{L^2(e)}^2 \left( \frac{\theta_2}{2\rho_2^e}+\frac{\bar{\gamma}_1}{2\rho_3^e}+\bar{\gamma}_2 \right).
 \end{align*}
With the choices $\rho_1^e = \frac{\varepsilon}{\theta_1}$, $\rho_2^e = \frac{\varepsilon}{\theta_2}$, and $\rho_3^e = 1$, terms further simplify such that 
\begin{align*}
 a(z,z;t) &\geq \frac{\varepsilon}{2} \lVert z \rVert_V^2 - \sum_{e \in \mathcal{E}^H} \lVert z^{e,p}\rVert_{L^2(e)}^2 \left( \frac{\theta_1^2}{2\varepsilon} +\frac{\bar{\gamma}_1}{2} + \frac{\varepsilon}{2}\right) \\ & \quad \! - \lVert z^{e,q}\rVert_{L^2(e)}^2 \left( \frac{\theta_2^2}{2\varepsilon}+\frac{\bar{\gamma}_1}{2}+\bar{\gamma}_2+\frac{\varepsilon}{2} \right)\\
 & \geq \frac{\varepsilon}{2} \lVert z \rVert_V^2 - \beta\lVert z \rVert_{H}^2
 \end{align*}
for $\beta\geq 0$ sufficiently small and $\alpha_e = \frac{\varepsilon}{2}$. Hence $a$ is weakly coercive as claimed.
\end{proof}

\paragraph{Proof of Lemma \ref{lemma:regularity-PDE}.}
We will first show that weak solutions to the homogenized problem exist and use this to obtain existence for the original problem. Then, we will show the a priori estimate \eqref{eq:a-priori-estimate}. Consider the homogenized problem: find $z:=y-w\in W(0,T;H,V_0)$ such that 
\begin{equation}\label{eq:varform-homogenized}
\begin{aligned}
  &  \int_0^T \left \langle z_t(t), \phi(t) \right \rangle_{V^*, V}\ + a(z(t),\phi(t);t) \D t \\
  &\quad = - \int_0^T \langle w_t(t), \phi(t)  \rangle_{V^*, V} +a(w(t);\phi(t);t)  \D t
  \end{aligned}
\end{equation}
for all $\phi \in W(0,T;H,V_0)$
and satisfying the initial condition $z_0:=y_0-w(\cdot,0)$. Let 
\[f(\phi):= - \int_0^T  \langle w_t(t), \phi(t)  \rangle_{V^*, V} +a(w(t);\phi(t);t) \D t.\]
We have by continuity of the bilinear form that
\begin{equation}
\label{eq:continuity-RHS}
\begin{aligned}
  |f(\phi)| &\leq  \int_0^T \lVert w_t(t)\rVert_{V^*} \lVert \phi(t) \rVert_V +K \lVert w(t)\rVert_V \lVert \phi(t)\rVert_V \D t\\
  & \leq (\lVert w_t\rVert_{L^2(0,T;V^*)} + K \lVert w\rVert_{L^2(0,T;V)}) \lVert \phi\rVert_{L^2(0,T;V)}
\end{aligned}
\end{equation}
and so
\begin{align*}
    \lVert f\rVert_{(L^2(0,T;V))^*} \leq  \lVert w_t\rVert_{L^2(0,T;V^*)} + K \lVert w\rVert_{L^2(0,T;V)}<\infty
\end{align*}
due to the assumption $w \in W(0,T;H,V)$. Since $(L^2(0,T;V))^* \cong L^2(0,T;V^*)$, the function $f$ belongs to $L^2(0,T;V^*)$. Recall that by Assumption~\ref{ass:PDE}, $y_0 \in H$, so that $z_0 \in H$. Now, Lemma~\ref{lem:propertiesbilinearform} provides the remaining ingredients to apply, for example,  \cite[Theorem 1.37]{Hinze2009}; see also \cite[Theorem 26.1]{Wloka1987}. From this result, we obtain the existence of the unique solution $z \in W(0,T; H,V_0)$ to \eqref{eq:varform-homogenized} with the given initial condition as well as its continuous dependence on the data $(z_0, f)$ with the following estimate (for some constant $C>0$):
\begin{align}
\label{eq:a-priori-homogenized}
    \|z \|^2_{L^2(0,t;V)} +  \|z_t \|^2_{L^2(0,t;V^*)} \leq C \left( \| z_0 \|^2_{H} + \| f \|^2_{L^2(0,t; V^*)} \right) \quad \text{a.e.~in } (0,T).
\end{align}

Now, it follows by the existence of $z$ that the weak solution $y=z+w \in W(0,T; H, V)$ to the original problem \eqref{eq:gas-model-detailed'} exists. For the estimate \eqref{eq:a-priori-estimate}, note that
\begin{equation}
\label{eq:LHS}
\begin{aligned}
 &
 \lVert y\rVert_{L^2(0,t;V)}^2 + \lVert y_t\rVert_{L^2(0,t;V^*)}^2\\ &\quad \leq 2 \left(
 \lVert z\rVert_{L^2(0,t;V)}^2 + \lVert w \rVert_{L^2(0,t;V)}^2 + \lVert z_t  \rVert_{L^2(0,t;V^*)}^2 + \lVert w_t \rVert_{L^2(0,t;V^*)}^2\right).
    \end{aligned}
\end{equation}
In the following, we use a generic positive constant $C$ that may take a different value at each appearance. From \eqref{eq:continuity-RHS}, \eqref{eq:a-priori-homogenized}, and \eqref{eq:LHS}, we have
\begin{equation}
\label{eq:apriori-intermediate-step}
\begin{aligned}
&
\lVert y\rVert_{L^2(0,t;V)}^2 + \lVert y_t\rVert_{L^2(0,t;V^*)}^2\\
& \quad \leq C \left( \| z_0 \|^2_{H} + \| f \|^2_{L^2(0,t; V^*)} + 
\lVert w \rVert_{L^2(0,t;V)}^2  + \lVert w_t \rVert_{L^2(0,t;V^*)}^2\right)\\
&\quad \leq  C \left( \| y_0 \|^2_{H}+ \| w_0 \|^2_{H} 
+ \lVert w \rVert_{L^2(0,t;V)}^2  + \lVert w_t \rVert_{L^2(0,t;V^*)}^2\right).
\end{aligned}
\end{equation}

Recall the affine linear construction of $w \in W(0,T;H,V)$: we have on any given edge $e=(\nu_1,\nu_2)$ that
\[w^{e,p}(x,t) = \frac{\ell_e-x}{\ell_e}  \mathfrak{p}^{\nu_1}(t) + \frac{x}{\ell_e} \mathfrak{p}^{\nu_2}(t)  \quad \text{and} \quad w^{e,q}(x,t) = \frac{\ell_e-x}{\ell_e} \mathfrak{q}^{\nu_1}(t) + \frac{x}{\ell_e}\mathfrak{q
}^{\nu_2}(t),\]
where $\mathfrak{p}^{\nu}$ and $\mathfrak{q}^{\nu}$ vanish for inner nodes $\nu \in \mathcal{V}_0^H$. Then
\begin{equation}
\label{eq:apriori-first-esimtate}
\begin{aligned}
\lVert w \rVert_{L^2(0,t;V)}^2 &= \int_0^t \lVert w(t) \rVert_{H}^2 + \lVert w_x(t) \rVert_H^2 \D t \\
&= \int_0^t \sum_{e \in \mathcal{E}^H} \lVert w^e(t) \rVert_{L^2(e)^2}^2 + \lVert w^e_x(t) \rVert_{L^2(e)^2}^2 \D t \\
&\leq C \int_0^t \sum_{\nu \in \mathcal{V}_\partial^H} \lVert \mathfrak{p}^{\nu}(t)\rVert_2^2 + \lVert \mathfrak{q}^{\nu}(t)\rVert_2^2 \D t\\
&= C \sum_{\nu \in \mathcal{V}_\partial^H} \lVert \mathfrak{p}^{\nu}\rVert^2_{L^2(0,t)} +\lVert \mathfrak{q}^{\nu}\rVert^2_{L^2(0,t)}
\end{aligned}
\end{equation}
and
\begin{equation}
\label{eq:apriori-second-esimtate}
\begin{aligned}
 & \lVert w_t \rVert_{L^2(0,t;V^*)}^2  = \int_0^t \lVert w_t(t) \rVert_{V^*}^2 \D t 
  = \int_0^t \sum_{e \in \mathcal{E}^H} \lVert w_t^{e,p}(t) \rVert_{H^{-1}(e)}^2 + \lVert w_t^{e,q}(t) \rVert_{H^{-1}(e)}^2  \D t\\
  & = \int_0^t \sum_{e \in \mathcal{E}^H} \left\lVert \frac{\ell_e-x}{\ell_e}  \mathfrak{p}^{\nu_1}_t(t) + \frac{x}{\ell_e} \mathfrak{p}_t^{\nu_2}(t) \right\rVert_{H^{-1}(e)}^2 + \left\lVert \frac{\ell_e-x}{\ell_e} \mathfrak{q}_t^{\nu_1}(t) + \frac{x}{\ell_e} \mathfrak{q}_t^{\nu_2}(t) \right\rVert_{H^{-1}(e)}^2  \D t\\
  & \leq C \int_0^t \sum_{e \in \mathcal{E}^H} \left\lVert \frac{\ell_e-x}{\ell_e}  \mathfrak{p}^{\nu_1}_t(t) + \frac{x}{\ell_e} \mathfrak{p}_t^{\nu_2}(t) \right\rVert_{L^2(e)}^2 + \left\lVert \frac{\ell_e-x}{\ell_e} \mathfrak{q}_t^{\nu_1}(t) + \frac{x}{\ell_e} \mathfrak{q}_t^{\nu_2}(t) \right\rVert_{L^2(e)}^2  \D t\\
  & \leq C \int_0^t \sum_{\nu \in \mathcal{V}_\partial^H}  \lVert \mathfrak{p}_t^{\nu}(t) \rVert_{2}^2 + \lVert \mathfrak{q}_t^{\nu}(t) \rVert_{2}^2  \D t,
 \end{aligned}
 \end{equation}
where the first inequality in \eqref{eq:apriori-second-esimtate} follows due to the embedding $L^2 \hookrightarrow H^{-1}$ and the fact that, by construction, $w_t^{e,p}(\cdot,t)$ and $w_t^{e,q}(\cdot,t)$ belong to $L^2(e)$ for all $t$ and $e$. The final inequality follows since there are a finite number of edges with finite lengths $\ell_e$. Since $\lVert y_0\rVert_{H}$ and $\lVert w_0\rVert_H$ are bounded, there exists $C>0$ large enough so that
\begin{equation}
    \label{eq:a-priori-initial-conditions}
   \lVert y_0\rVert_{H}^2 \leq C\lVert w_0\rVert_H^2.
\end{equation}

It now follows from \eqref{eq:apriori-first-esimtate}, \eqref{eq:apriori-second-esimtate}, and \eqref{eq:a-priori-initial-conditions} that \eqref{eq:apriori-intermediate-step} can be bounded as follows:
\begin{equation*}
\lVert y\rVert_{L^2(0,t;V)}^2 + \lVert y_t\rVert_{L^2(0,t;V^*)}^2 \leq  C \left( \| y_0 \|^2_{H} + \sum_{\nu \in \mathcal{V}_\partial^H} (\lVert\mathfrak{p}^{\nu}\rVert^2_{H^1(0,t)} +\lVert\mathfrak{q}^{\nu}\rVert^2_{H^1(0,t)} )\right).  
\end{equation*}
With that, the proof of Lemma \ref{lemma:regularity-PDE} is complete.

\subsection{Proof of existence of equilibria}
\label{sec:existence-equilibria}
An immediate consequence of \eqref{eq:a-priori-estimate} is the continuity of the (linear) map
\begin{align*}
S\colon H\times H^1(0,T)^{n_h}\times H^1(0,T)^{n_h} \rightarrow W(0,T;H, V): \quad (y_0,\mathfrak{p},\mathfrak{q}) \mapsto y,
\end{align*}
provided $\mathfrak{p}$ and $\mathfrak{q}$ are compatible with $y_0$. Letting $ \mathfrak{p}^\nu  = \sum_{i=1}^N \hat{p}_{i\nu}$ and $\mathfrak{q}^\nu =\sum_{i=1}^N \hat{q}_{i\nu}$, and substituting the variable $\hat{q}_{i\nu}$ by means of the definitions \eqref{eq:conversion-conditions}, we define the continuous affine operator
\begin{align*}
G_H \colon  H^1(0,T)^{n_g N}\times H^1(0,T)^{n_s N} \times H^1(0,T)^{n_c N}\times H^1(0,T)^{n_h N} &\rightarrow W(0,T;H, V),\\
(g,s,c,\hat{p})  &\mapsto y.
\end{align*}
Generation and sales can also occur on the electricity network, but to keep notation to a minimum, we do not differentiate between the decision variables here.
We denote the closed convex constraint set associated with \eqref{eq:gas-model-detailed-gas-9} by
\begin{align*}
K_H: = \{ (p,q) \in W(0,T;H,V) \mid \, & 0 \leq q^e(x,t) \leq q_e^{\max}, \\ &0 \leq p^e(x,t) \leq p_e^{\max} \, \text{a.e.~in } e\times (0,T) \, \forall e \in \mathcal{E}^H
\}.
\end{align*}
With these definitions, we can write the full set of constraints on the gas network as
\[
G_H(g, s, c,\hat{p}) \in K_H.
\]

On the electricity network, the
    constraints depend linearly on the decision variables $s, c, g$. For  \eqref{eq:electricity-model-detailed-1}--\eqref{eq:electricity-model-detailed-2}, it is clear that 
there exists a continuous affine linear operator 
\[
G_E \colon H^1(0,T)^{n_g N}\times H^1(0,T)^{n_sN} \times H^1(0,T)^{n_cN} \rightarrow H^1(0,T)^{N}\times H^1(0,T)^{|\mathcal{E}^E|}
\]
such that \eqref{eq:electricity-model-detailed-1}--\eqref{eq:electricity-model-detailed-2} can be written as 
\[G_E(g,s,c) \in K_E,\] 
where \[K_E := \{(k_1,k_2) \in H^1(0,T)^{N}\times H^1(0,T)^{|\mathcal{E}^E|} \mid k_1=\{0\}\text{ and } k_2 \leq 0 \text{ a.e.~in } (0,T)\}\] is a closed convex cone.

Using the operators defined above, we can write problem \eqref{eq:gas-model-detailed} in the form 
\begin{subequations}
\label{eq:gas-model-reduced-1}
\begin{alignat}{3}
    & \max_{u_i \in U_{\textup{ad}, i}} \;\; &&f_i(u_i, u_{-i}) \hspace{3cm} &&\text{Agent's objective \eqref{eq:gas-model-detailed-obj}} \label{eq:gas-model-1} \\
    & \text{s.t.} && G_E (g,s,c) \in K_E, && \text{Electricity constraints \eqref{eq:electricity-model-detailed-1}--\eqref{eq:electricity-model-detailed-2}} \label{eq:gas-model-2} \\
    &  &&G_H (g,s,c,\hat{p}) \in K_H, && \text{Gas constraints \eqref{eq:gas-model-detailed-gas-1}--\eqref{eq:gas-model-detailed-gas-9}} \label{eq:gas-model-3}
    \end{alignat}
\end{subequations}
for each agent $i$ in the set of firms $\{1, \dots, N\}$. Defining the operator $G:=(G_E, G_H)$, the set $K:=K_E \times K_H$, and the objective $\tilde{f}_i:=-f_i$, we can reduce problem~\eqref{eq:gas-model-reduced-1} to a generalized Nash equilibrium problem, where the agents' problems take the form
\begin{equation}
    \label{eq:the-game-reduced}
    \begin{aligned}
    &\min_{u_i \in U_{\textup{ad},i}} \tilde{f}_i(u_i,u_{-i}) \\
    &\text{s.t. }G (u_i,u_{-i}) \in K.
    \end{aligned}
\end{equation}
We have transformed the concave objective to a convex one by switching the sign.
We denote the induced joint constraint set for all agents by 
\begin{equation}
\label{eq:feasible-set}
\mathfrak{F}:=\{ u \in U_{\rm ad} \mid G(u_i,u_{-i}) \in K \}.
\end{equation}
Points belonging to this set are called feasible. Moreover, given a point $u_{-i}$, we write
\[
\mathfrak{F}_i(u_{-i}) :=\{ u_i \in U_i \mid (u_i,u_{-i}) \in \mathfrak{F} \}
\]
for the feasible set of agent $i$'s problem. We recall that a feasible collective decision $\bar{u} \in \mathfrak{F}$ is called a \textit{generalized Nash equilibrium} if 
\begin{equation*}
    \tilde{f}_i(\bar{u}_i,\bar{u}_{-i}) \leq \tilde{f}_i(v_i, \bar{u}_{-i})\quad \forall v_i \in \mathfrak{F}_i(\bar{u}_{-i}). 
\end{equation*} 
We show the existence of an equilibrium by first introducing the Nikaido--Isoda function $\Psi\colon U \times U \rightarrow \R$, which is defined by
\[
\Psi(u,v) = \sum_{i=1}^N \Big( \tilde{f}_i({u}_i,{u}_{-i}) - \tilde{f}_i(v_i, {u}_{-i})\Big).
\]


We now collect some results for our problem. 
\begin{lemma}
\label{lemma:weak-compactness}
Under Assumption \ref{ass:SGNEP}, the set $\mathfrak{F}$ defined by \eqref{eq:feasible-set} is weakly compact.
\end{lemma}
\begin{proof}
Consider the nontrivial case where $\mathfrak{F} \neq \emptyset$. Then, by Assumption~\ref{ass:private-restrictions},  
$U_{\textup{ad}}$ is a nonempty, bounded, closed, and convex subset of the Hilbert space
\[
U = \prod_{i=1}^N U_i, \quad \text{where} \quad  U_i= H^1(0,T)^{n_g} \times H^1(0,T)^{n_s} \times H^1(0,T)^{n_c} \times  H^1(0,T)^{n_h}.
\]

The space $U$ is reflexive, making $U_{\textup{ad}}$ weakly sequentially compact. We note that $(u_i,u_{-i}) \mapsto G(u_i,u_{-i})$ is weakly continuous since it is a continuous affine linear operator; cf.~\cite[Lemma 2.34]{Bauschke2011}. 
Now, consider a sequence $\{ u^n\} \subset \mathfrak{F}$ such that $u^n \rightharpoonup u$. Then for every $n$, we have $u^n \in U_{\rm ad}$ and $G(u^n_i, u^n_{-i}) \in K$. Since $G$ is weakly continuous, it follows that $\mathfrak{F}$ is weakly closed; since $U_{\textup{ad}}$ is bounded, $\mathfrak{F}$ is indeed weakly compact as claimed.
 \end{proof}

\begin{theorem} 
\label{thm:NI-function-lsc}
    The Nikaido--Isoda function is weakly sequentially lower semicontinuous with respect to $u$.
\end{theorem}
\begin{proof}
Notice that the function
\begin{align*}
(u_i,u_{-i}) \mapsto \int_0^T \Big[ &\sum_{\nu \in \mathcal{V}^s} s_{i\nu}(\tau)\left( \bar{a}_\nu(\tau)-\bar{b}_\nu(\tau)\sum_{j=1}^Ns_{j\nu}(\tau)\right) \\
  &-\sum_{\nu' \in \mathcal{V}^g} g_{i\nu'}(\tau) \kappa_{i\nu'}     -\sum_{\nu'' \in \mathcal{V}^{\textup{PtG}}}  c_{i\nu''}(\tau)  \kappa_{i\nu''}  \Big] 
\D \tau  
\end{align*}
is concave, making the mapping $(u_i,u_{-i}) \mapsto \tilde{f}_i(u_i,u_{-i})$ convex and therefore (weakly) lower semicontinuous. 
Also, $\tilde{f}_i(v_i,\cdot)$ is linear. Together, we obtain the (weak) lower semicontinuity of $u \mapsto \Psi(u,v).$ 
\end{proof}
Provided \textit{$\mathfrak{F}$ is nonempty}, the existence of a generalized Nash equilibrium follows now by the following result from \cite[Theorem 2.3]{kanzow2019multiplier}.
\begin{theorem}
  Suppose the Nikaido--Isoda function is weakly sequentially lower semicontinuous with respect to $u$ and assume that $\mathfrak{F}$ is nonempty and weakly compact. Then a generalized Nash equilibrium exists.
\end{theorem}

\section{Conclusion}
\label{sec:conclusion}
In this paper, we presented a novel Cournot-Nash model of an intraday hydrogen and electricity market that is coupled by means of conversion stations between the two energy sources.  The model can be mathematically represented as a generalized Nash equilibrium problem since agents must collectively satisfy operational constraints on the network. The main result was showing the existence of an equilibrium to this GNEP. To this end, we needed to make several simplifying assumptions to ensure the concavity of the agents' problems; investigating GNEPs with nonlinear constraints is a challenging topic of future research. Beyond this, we intend to analyze the related game in which an independent agent does not have complete information about the parameters in the networks or the decisions of other agents. In reality, the decision process is in some sense distributed among economic agents since there is no entity that can find the equilibrium and then prescribe each agent its optimal strategy $\bar{u}_i$. 

The numerical solution of the GNEP and its distributed version is also planned for future studies. The state constraints \eqref{eq:gas-model-detailed-gas-9} from the gas network can be handled using a path-following approach as in \cite{Hintermueller2015}. This involves solving a sequence of regularized problems and carefully increasing the regularization parameter to enforce the state constraints. Using their optimality conditions, the regularized problems can be considered as generalized equations/inclusions or reformulated as variational inequalities, which can be handled using gradient-descent-type or higher-order methods. 

To handle the distributed aspects, ideas from centralized distributed optimization may be fruitful in our setting. Here, an additional agent called the ``operator" is introduced whose goal is to establish the work of the whole physic-economic system. For that purpose, it knows the parameters of the networks and inverse demand functions and may observe the actions of the agents in the network, i.e., sales, gas flows, pressures, etc. At the same time, it does not have access to the parameters of individual agents' objectives due to privacy. Under these information constraints, it is possible, in particular, to implement extragradient-type methods \cite{beznosikov2022decentralized,trandinh2024revisiting} for solving the system of optimality conditions for the regularized version of our GNEP. The operator observes the actions of the agents, solves the PDE, and communicates the derivative of the penalty for the state constraint back to the agents, thus, imposing the regulations that allow the network to operate properly. The agents, on the other hand, observe their revenue (that depends on the actions of the other agents) and its derivative and update their strategies by a combination of that derivative with the derivative of the penalty for state constraints, resulting in combining two goals: increase their revenue and ensure the proper operation of the networks. This point of view serves two goals: it allows to propose a numerical algorithm for finding an equilibrium taking into account information constraints and it provides a way for agents to put a dynamic perspective on their strategies so that they update their strategies and this dynamic process goes to equilibrium strategies.
The main challenge will be to establish the correctness and convergence rate of this procedure in the infinite-dimensional case.

\backmatter
\bmhead{Acknowledgements}
The authors are thankful for support from the Deutsche Forschungsgemeinschaft (DFG, German Research Foundation) under Germany's Excellence Strategy – The Berlin Mathematics Research Center MATH+ (EXC-2046/1, project ID: 390685689). The authors also thank the Deutsche Forschungsgemeinschaft [Projects B02 and B09 in the “Sonderforschungsbereich/Transregio 154 Mathematical Modelling, Simulation and Optimization Using the Example of Gas Networks”] for support.

\begin{appendices}
\section{Function spaces}
\label{sec:function-spaces}
Here, we will show the claimed properties of $H$ and $V$ which were introduced at the beginning of Section \ref{sec:existence}. We recall that for $I \subset \R$, the sets $L^2(I)$ and $H^1(I)$ are separable Hilbert spaces with the corresponding inner products
\[
(f,g)_{L^2(I)} = \int_I f(x) g(x) \D x \quad \text{and} \quad (f,g)_{H^1(I)} = (f',g')_{L^2(I)} + (f,g)_{L^2(I)},
\]
and satisfying the Gelfand triple structure $H^1(I) \hookrightarrow L^2(I) \hookrightarrow H^{-1}(I)$. Therefore, the products $H= (\prod_{e \in \mathcal{E}^H} L^2(e))^2$ and $H':=(\prod_{e \in \mathcal{E}^H} H^1(e))^2$ are also separable Hilbert spaces when equipped with the scalar products given in \eqref{eq:scalar-product1} and \eqref{eq:scalar-product2}, respectively. 

Recall the definition \eqref{eq:V0-definition}; we now argue that $(V_0, (\cdot,\cdot)_V)$ is a separable Hilbert space; it will be enough to show that it is a closed linear subspace of $H'$, since then separability of $V_0$ follows (cf.~\cite[Theorem 1.22]{Adams2003}). Let $\phi_n =(\phi_n^p,\phi_n^q)\in V_0$ be a sequence with $\phi_n \rightarrow \phi$ in $H'$ as $n\rightarrow \infty$. We wish to show that $\phi=(\phi^p,\phi^q) \in V_0$, i.e., $\phi^e(\nu,t)=(0,0)$ for all $\nu \in \mathcal{V}_\partial^H$ and $e \in k(\nu)$ as well as the coupling conditions $\phi^{e,p}(\nu,t) = \phi^{e',p}(\nu,t)$ and $\sum_{e\in k(v)}\phi^{e,q}(\nu,t)n^e(\nu) = 0$  for all $\nu \in \mathcal{V}_\partial^H$ and $e \in k(\nu)$. We show the argument for the conditions on $\mathcal{V}_\partial^H$. Taking sequences $x_n \rightarrow \nu$, we have 
\begin{equation*}
    \sum_{\nu \in \mathcal{V}_\partial^H} \lVert \phi(\nu,t)\rVert_2 \leq  \sum_{\nu \in \mathcal{V}_\partial^H} \lVert \phi(\nu,t)-\phi_n(\nu,t)\rVert_2 + \lVert \phi_n(\nu,t)-\phi_n(x_n,t)\rVert_2 + \lVert \phi_n(x_n,t)-\phi_n(\nu,t)\rVert_2,
\end{equation*}
the right-hand side of which converges to zero as $n\rightarrow \infty$, where we used the fact that $H^1(e)$ is continuously embedded into $C([0,\ell_e])$ for all $e$. Analogous arguments can be applied to the other conditions, and we can conclude that $V$ is also a separable Hilbert space. As a final remark, $\tilde{V} \in \{V,V_0 \}$ satisfies the Gelfand triple structure $\tilde{V} \hookrightarrow H \hookrightarrow \tilde{V}^*$. 
\end{appendices}

\bibliographystyle{plain}
\bibliography{references}

\end{document}